\documentclass{article}
\usepackage[utf8]{inputenc}
\usepackage[margin=1in]{geometry}
\usepackage[affil-it]{authblk} 
\usepackage{etoolbox}
\usepackage{lmodern}

\makeatletter
\patchcmd{\@maketitle}{\LARGE \@title}{\fontsize{16}{19.2}\selectfont\@title}{}{}
\makeatother

\newcommand{\indic}{\mathbbm{1}}
\newcommand{\tsim}{\mbox{\tiny$\sim$}}

\newcommand{\caseone}{High collision regime~}
\newcommand{\casetwo}{Low collision regime~}

\title{Convergence of Chao Unseen Species Estimator}
\author{%
  Nived Rajaraman, Prafulla Chandra, Andrew Thangaraj\\
  Department of Electrical Engineering\\
  Indian Institute of Technology Madras\\
  Chennai 600036, India\\
  Email: \{ee14b040,ee16d402,andrew\}@ee.iitm.ac.in
  \and
  Ananda Theertha Suresh\\
  Google Research\\
  New York, USA\\
  Email: \{theertha\}@google.com
}

\usepackage{amsfonts}
\usepackage{times}
\usepackage{latexsym}
\usepackage{amssymb}
\usepackage{amsmath}
\usepackage{enumitem}
\usepackage{cancel}
\usepackage{amsthm}
\usepackage{vector}
\usepackage{mathtools}
\usepackage[
            CJKbookmarks=true,
            bookmarksnumbered=true,
            bookmarksopen=true,
                       bookmarks=true,
            colorlinks=true,
            citecolor=red,
            linkcolor=blue,
            anchorcolor=red,
            urlcolor=blue
            ]{hyperref}
\usepackage{verbatim}
\usepackage{bbm}
\usepackage{graphicx,subcaption}
\usepackage{fancyhdr}
\usepackage{xcolor}

\usepackage{mathrsfs}
\usepackage{pgfplots}
\pgfplotsset{compat=1.14}
\usepackage{tikz}
\usetikzlibrary{positioning,chains,fit,shapes,calc}
\tikzstyle{bipartite}=[thick,
  leftnode/.style={fill=blue,circle,draw,inner sep=0pt,minimum size=2mm},
  rightnode/.style={fill=green,circle,draw,inner sep=0pt,minimum size=2mm},
  dummynode/.style={fill=gray,circle,draw,inner sep=0pt,minimum size=2mm},
  ->,shorten >= 3pt,shorten <= 3pt
]
\def\addlegendimage{\csname pgfplots@addlegendimage\endcsname}

\pgfkeys{/pgfplots/number in legend/.style={%
        /pgfplots/legend image code/.code={%
            \node at (0.295,-0.0225){#1};
        },%
    },
}

\tikzstyle{centerlabel}=[fill=white,
anchor=center,
midway]

\newtheorem{theorem}{Theorem}
\newtheorem{definition}[theorem]{Definition}

\newtheorem{lemma}[theorem]{Lemma}

\theoremstyle{remark}

\newcommand{\sigmaAC}{\sigma_{\mathrm{Chao}}}

\newcommand{\poly}{\mathrm{poly}}
\newcommand{\linear}{\mathrm{linear}}

\newcommand{\diff}{\mathrm{d}}
\newcommand{\EE}{\mathbb{E}}

\def\bb0{{\mathbb{0}}}

\def\bb{{\mathbf{b}}}

\def\b0{{\mathbf{0}}}

\def\b1{{\mathbf{1}}}

\def\bbE{{\mathbb{E}}}

\def\bbN{{\mathbb{N}}}

\def\cE{\mathcal{E}}

\def\cX{\mathcal{X}}

\def\sf0{{\mathsf{0}}}

\allowdisplaybreaks

\usepackage[backend=bibtex8, style=alphabetic]{biblatex}
\bibliography{refs}

\begin{document}

\maketitle

\footnote{This work was presented in part at the 2019 IEEE International Symposium on Information Theory (ISIT 2019) \cite{ISIT19}}

\begin{abstract}
Support size estimation and the related problem of unseen species estimation have wide applications in ecology and database analysis. Perhaps the most used support size estimator is the Chao estimator. Despite its widespread use, little is known about its theoretical properties. We analyze the Chao estimator and show that its worst case mean squared error (MSE) is smaller than the MSE of the plug-in estimator by a factor of $\mathcal{O}((k/n)^4)$, where $k$ is the maximum support size and $n$ is the number of samples. Our main technical contribution is a new method to analyze rational estimators for discrete distribution properties, which may be of independent interest.
\end{abstract}

\section{Introduction}

\noindent Given independent samples from an underlying unknown distribution, we consider the problem of estimating the support size of the distribution. Estimating the support size and unseen species estimation has applications in ecological diversity~\cite{C84,SV84,SCL03,C05,CCG12}, vocabulary size estimation~\cite{ET76,TE87}, database attribute variation~\cite{HNSS95}, password analysis~\cite{FH07}, and, recently, in modern applications such as microbial diversity~\cite{HHJTB01,PBG01,GCB07} and genome sequencing~\cite{daley2013predicting}.

Formally, let $P$ denote the unknown distribution over domain $\mathcal{X}$. Upon observing $N$ independent samples $X_1,X_2,\ldots, X_N \stackrel{\text{def}}{=} X^N$ from $P$, the goal is to estimate the support size,
\begin{equation*}
S(P) \stackrel{\text{def}}{=} \sum_{x \in \mathcal{X}} \mathbbm{1}_{p_x > 0}.
\end{equation*}
Let $N_x(X^N)$ be the number of occurrences of symbol $x$ in $X^N$. The simplest estimator is the \emph{plug-in} or the \emph{empirical} estimator, which estimates $S(P)$ by
\begin{equation}
 \label{eq:plug_def}
\hat{S}^\text{pl}(X^N) \stackrel{\text{def}}{=} \sum_{x \in \mathcal{X}} \mathbbm{1}_{N_x(X^N) > 0}.
\end{equation}
The plug-in estimator often performs poorly in the non-asymptotic regime, where $N \approx S(P)$. To overcome this, several estimators have been proposed, including the Efron-Thisted estimator~\cite{ET76}, the Chao estimator~\cite{C84}, and, more recently, a near-optimal estimator via linear programming \cite{VV11,VV13} and an optimal linear estimator via Chebyshev polynomials~\cite{WY14b}.

Of the above, perhaps the most used estimator is the Chao estimator which has seen wide usage in ecological~\cite{C84} and microbiological~\cite{LEMOS201142} applications among others. Despite its widespread use, apart from the analysis of the expectation of the estimator in the original paper~\cite{C84}, not much is known about its theoretical properties. In this paper, we analyze the Chao estimator and provide bounds on its worst case \emph{mean squared error} (MSE). In the next section, we state the problem definition and the statistical model.

\subsection{Preliminaries and Notation}

In general, support size estimation is an ill-posed problem as there might be a large set of symbols with infinitesimally small probability, which can never be detected with any finite number of samples. To overcome this, following~\cite{RRSS09, VV11, WY14b}, we focus on distributions where every non-zero probability is lower-bounded. Formally, we restrict ourselves to $\Delta_k$, the set of distributions such that all non-zero symbols have probability $\geq 1/k$. By the law of total probability, distributions in $\Delta_k$ have support size upper-bounded by $k$.

Support size estimation has been studied in a number of different statistical models, including multinomial~\cite{GT56}, Poisson, and Bernoulli-product models~\cite{CCG12}.
Following~\cite{C84, orlitsky2016optimal}, we study the problem in the Poisson sampling model, where the number of observed samples $N$ is a Poisson random variable with known mean $n$. Under Poisson sampling, the multiplicities of symbols $N_x(X^N)$, $x\in\mathcal{X}$, are independent random variables, and $N_x(X^N)$ is Poisson with mean $n p_x$. The 
independence of multiplicities comparatively simplifies the MSE analysis. We believe similar results should hold for the other above stated statistical models.

For a distribution $P$ and an estimator $\hat{S}(X^N)$, we measure the performance of the estimator in terms of MSE, given by
\begin{equation}
\label{eq:spmse}
\cE_n(\hat{S}, P) \stackrel{\text{def}}{=} \bbE_{X^N \sim P} (S(P) - \hat{S}(X^N))^2,
\end{equation}
and the worst case MSE over all distributions is
\begin{equation*}
\cE_{n,k}(\hat{S}) \stackrel{\text{def}}{=} \max_{P \in \Delta_k} \cE_n(\hat{S}, P).
\end{equation*}
The simple plug-in estimator only takes into account the number of seen symbols and does not try to predict the symbols that are not observed yet. In this context, Efron-Thisted~\cite{ET76} and Chao~\cite{C84}, observed that support size estimation is closely related to the problem of unseen species estimation, where the goal is to estimate the number of
symbols that have not yet appeared and will appear in the future,
\begin{equation*}
U(X^N, P) \stackrel{\text{def}}{=} \sum_{x \in \mathcal{X}} \mathbbm{1}_{p_x > 0}
\mathbbm{1}_{N_x = 0}.
\end{equation*}
Given an estimator $\hat{U}(X^N)$ for $U(X^N, P)$, one can estimate the support size via
\begin{equation}
\label{eq:utos}
\hat{S}^{\text{pl}}(X^N) + \hat{U}(X^N).
\end{equation}
Let the \emph{prevalence} or \emph{finger-print} $\varphi_i(X^N)$ denote the number of symbols with non-zero probability that appeared $i$ times. For $i \geq 1$
\begin{equation*}
\varphi_i(X^N) \stackrel{\text{def}}{=} \sum_{x \in \mathcal{X}} \mathbbm{1}_{N_x = i},
\end{equation*}
and, for $i = 0$, $\varphi_0(X^N,P) \stackrel{\text{def}}{=} \sum_{x \in \mathcal{X}} \mathbbm{1}_{N_x = 0} \mathbbm{1}_{p_x >0}$. With this notation, $S(P) = \varphi_0(X^N,P)+\sum_{i \geq 1} \varphi_i(X^N)$, the plug-in estimator, $\hat{S}^{\text{pl}} = \sum_{i \geq 1} \varphi_i(X^N)$, and $U(X^N, P) = \varphi_0(X^N,P)$. Hence, for estimators of the form ~\eqref{eq:utos},
\begin{equation*}
S(P) -\hat{S} = \varphi_0(X^N,P) - \hat{U}(X^N),
\end{equation*}
and the error in estimating the support is same as the error in estimating the unseen symbols. Similar to \eqref{eq:spmse}, we define the worst case mean squared error in estimating the unseen symbols by
\begin{equation*}
\cE_{n,k}(\hat{U}) = \max_{P \in \Delta_k} \bbE_{X^N \sim P} (\hat{U}(X^N) - \varphi_0(X^N,P))^2,
\end{equation*}
and, hence, for the support estimator $\hat{S} = \hat{S}^{\text{pl}} + \hat{U}$,
\begin{equation*}
\cE_{n,k}(\hat{S}) = \cE_{n,k}(\hat{U}).
\end{equation*}
Chao~\cite{C84} proposed the following estimator to estimate the number of unseen symbols\footnote{We use $N_x$ and $\varphi_i$ to abbreviate $N_x(X^N)$ and $\varphi_i(X^N)$ for simplicity.},
\begin{equation*}
\hat{U}^{\text{c}}(X^N) = \frac{\varphi^2_1}{2 \varphi_2},
\end{equation*}
which has a rational form and is not in the class of linear estimators. To understand the Chao estimator, first observe that $\varphi_i = \sum_{x \in \mathcal{X}} \mathbbm{1}_{N_x = i}$. Since $N_x$ is a Poisson random variable with mean $np_x$, 
\begin{equation*}
\bbE[\varphi_i] = \sum_{x \in \mathcal{X}} e^{-np_x} \frac{(np_x)^i}{i!}.
\end{equation*}
By the Cauchy-Schwarz inequality,
\begin{align}
\bbE[\varphi_0] \cdot \bbE[\varphi_2] & = \left(\sum_{x \in \mathcal{X}} e^{-np_x} \right) \cdot \left(\sum_{x \in \mathcal{X}} e^{-np_x} \frac{(np_x)^2}{2!} \right), \nonumber \\
& \geq \left(\sum_{x \in \mathcal{X}} e^{-np_x} \frac{(np_x)}{\sqrt{2!}} \right)^2 = \frac{(\bbE[\varphi_1])^2}{2}. \label{eq:cauchy}
\end{align}
Hence, $\frac{\bbE[\varphi_1]^2}{2\bbE[\varphi_2]} \leq \bbE[\varphi_0]$, and thus is a lower bound on the expected number of unseen symbols. Since expectations are not available, Chao \cite{C84} proposed to use $\frac{\varphi^2_1}{2\varphi_2}$ as an estimator for $\varphi_0$.

\section{Main Results}

Before we state results for the Chao estimator, we first state a folklore result on the performance of the plug-in estimator. 
\begin{lemma}
\label{lem:plugin}
For the plug-in estimator $\hat{S}^{\text{pl}}$ defined in ~\eqref{eq:plug_def}, 
\begin{equation*}
k^2 e^{-2n/k} + k e^{-n/k} \ge \cE_{n,k}(\hat{S}^{\text{pl}}) \ge k^2 e^{-2n/k} + k e^{-n/k} - k e^{-2n/k}.
\end{equation*}
\end{lemma}
\begin{proof}
For any distribution $p \in \Delta_k$, let
\begin{align*}
S(P) - \hat{S}^{\text{pl}} = \varphi^0 = \sum_{x \in \cX} \indic_{N_x = 0}.
    \end{align*}
Hence,    
    \begin{align*}
 \EE[(S(P) - \hat{S}^{\text{pl}}   )^2] 
 & = \EE[(\sum_{x \in \cX} \indic_{N_x = 0})^2] \\
 & \stackrel{(a)}{=} \EE^2[(\sum_{x \in \cX} \indic_{N_x = 0})] + \text{Var}\left(\sum_{x \in \cX} \indic_{N_x = 0}\right)\\
 & \stackrel{(b)}{=} \EE^2[(\sum_{x \in \cX} \indic_{N_x = 0})] + \sum_{x \in \cX} \text{Var}( \indic_{N_x = 0}) \\
 & \stackrel{(c)}{=} (\sum_{x \in \cX} e^{-np_x})^2 + \sum_{x \in \cX} e^{-np_x} (1-e^{-np_x}),
    \end{align*}
    where $(a)$ follows from the definition of bias and variance, $(b)$ follows from the fact that variance of sum of independent random variables is the sum of variance of independent random variables, and $(c)$ follows from the fact that $\indic_{N_x = 0}$ is a Bernoulli random variable with parameter $e^{-np_x}$. The lower bound follows by substituting $p$ to be the uniform distribution over $k$ elements and the upper bound follows by the convexity of the function $p \to e^{-np}$.
\end{proof}
\noindent Observe that the Chao estimator is undefined if $\varphi_2 = 0$. To circumvent this, we consider the closely related \emph{modified Chao estimator},
\begin{equation*}
\hat{U}^{\text{mc}} (X^N) = \frac{\varphi_1^2}{2 (\varphi_2 + 1)}.
\end{equation*}
The analysis of MSE for the Chao estimator and the modified Chao estimator are involved, as they are rational functions over the prevalences. Furthermore, the prevalences are dependent on each other. By developing new tools to analyze the expectation of ratios of functions of prevalences, we show the following.
\begin{theorem} 
\label{theorem:main:MSE}
For the modified Chao estimator,
	\begin{equation*}
\cE_{n,k} (\hat{U}^{\text{mc}}) \le k^2 \left( \frac{1}{1 + n/(k\alpha)}\right)^4 e^{-2n/k} + \epsilon(n,k),
	\end{equation*}
	where $\alpha=0.5569...$ solves $u^2=4e^{-2}e^{-u}$ and
	\begin{align*}
	\epsilon(n,k) = \left( \frac{4k^4}{\left(n^{4/5} - \sqrt{4/\pi}\right)^3} + \frac{(32.28)k^4}{n^{12/5}} + \frac{(98.97)k^3}{n^{11/5}} + \frac{2k^2}{n^{6/5}} + \frac{(1.77)k}{n^{1/5}} + \frac{(22.21)k^2}{n^2} \right).
	\end{align*}
\end{theorem}
For the non-asymptotic regime of interest, where $n = \Omega(k)$, $\epsilon(n,k)$ is $o(k^2)$ and the first term dominates. Hence, for $n = \Omega(k)$, the Chao estimator has better worst case MSE than the plug-in estimator. Furthermore, when $n \geq k$, the worst case MSE of the Chao estimator is at least a factor $(k/n)^4$ lower than the worst case MSE of the plug-in estimator (Lemma~\ref{lem:plugin}) and, for $n \ll k$, the worst case performance of the Chao estimator approaches that of the plug-in estimator. 

We note that the best estimator for support size and the unseen species problem achieves the worst case MSE
\begin{equation*}
\min_{\hat{S}}\ \cE_{n,k}(\hat{S}) = k^2 \cdot \exp \left(- \Theta \left( \sqrt{\frac{n \log k}{k}} \vee
\frac{n}{k} \vee 1 \right) \right),
\end{equation*}
and is achieved by the \emph{Chebyshev linear estimator}~\cite{WY14b}, obtained by the approximation properties of Chebyshev polynomials.

An empirical comparison of three estimators: plug-in, Chao, and Chebyshev estimators on synthetic data is shown in Fig.~\ref{fig:1}. The Chebyshev estimator is
parameterized by constants $c_0$ and $c_1$, which we choose as $0.45$ and $0.5$ as suggested in \cite{WY14b}. The distributions are chosen from $\Delta_k$ with $k = 10^4$. We consider (i) the uniform distribution on $k$ symbols, (ii) the $\mathrm{Zipf}(1)$ distribution with probability of the $i^{th}$ symbol proportional to $i^{-1}$, (iii) the geometric distribution with probability of the $i^{th}$ symbol proportional to $\alpha^{i-1}$ where $\alpha = 1 - k^{-1}$, and (iv) an even mixture of two uniform distributions, with probability of half of the symbols as $k^{-1}$ and the other half as $3 k^{-1}$. From Fig. \ref{fig:1}, the convergence rate of the modified Chao estimator is seen to be higher than the plug-in estimator over the distributions we considered. However, with the exception of the uniform distribution, the Chebyshev estimator outperforms the modified
Chao estimator. In the rest of the paper, we provide a proof of Theorem \ref{theorem:main:MSE}.
\begin{figure}[ht]
\begin{subfigure}[h]{0.5\textwidth}
                \includegraphics[width=\linewidth]{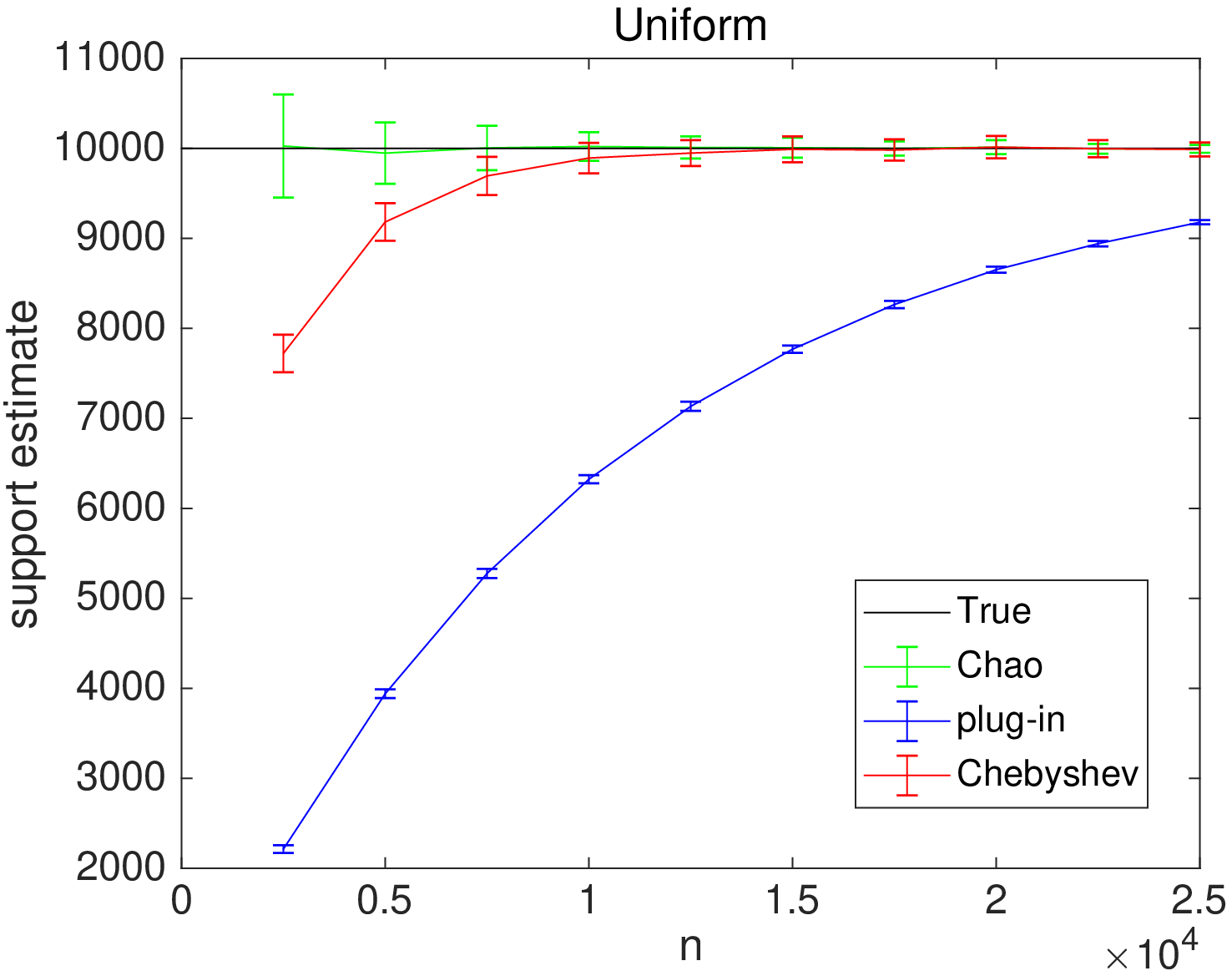}
                \label{fig:uniform}
        \end{subfigure}%
        \begin{subfigure}[h]{0.5\textwidth}
                \includegraphics[width=\linewidth]{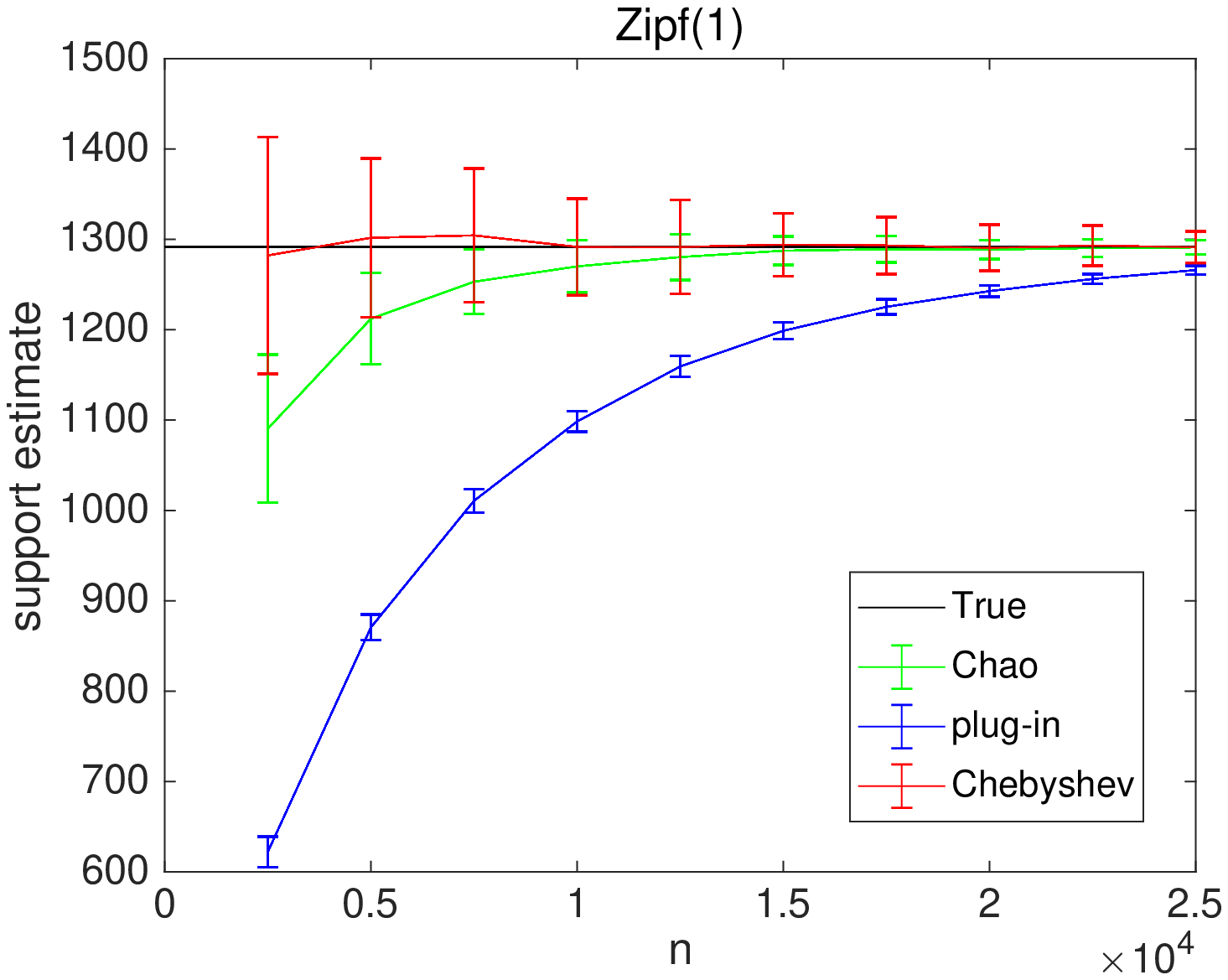}
                \label{fig:zipf1}
		\end{subfigure}\\%
        \begin{subfigure}[h]{0.5\textwidth}
                \includegraphics[width=\linewidth]{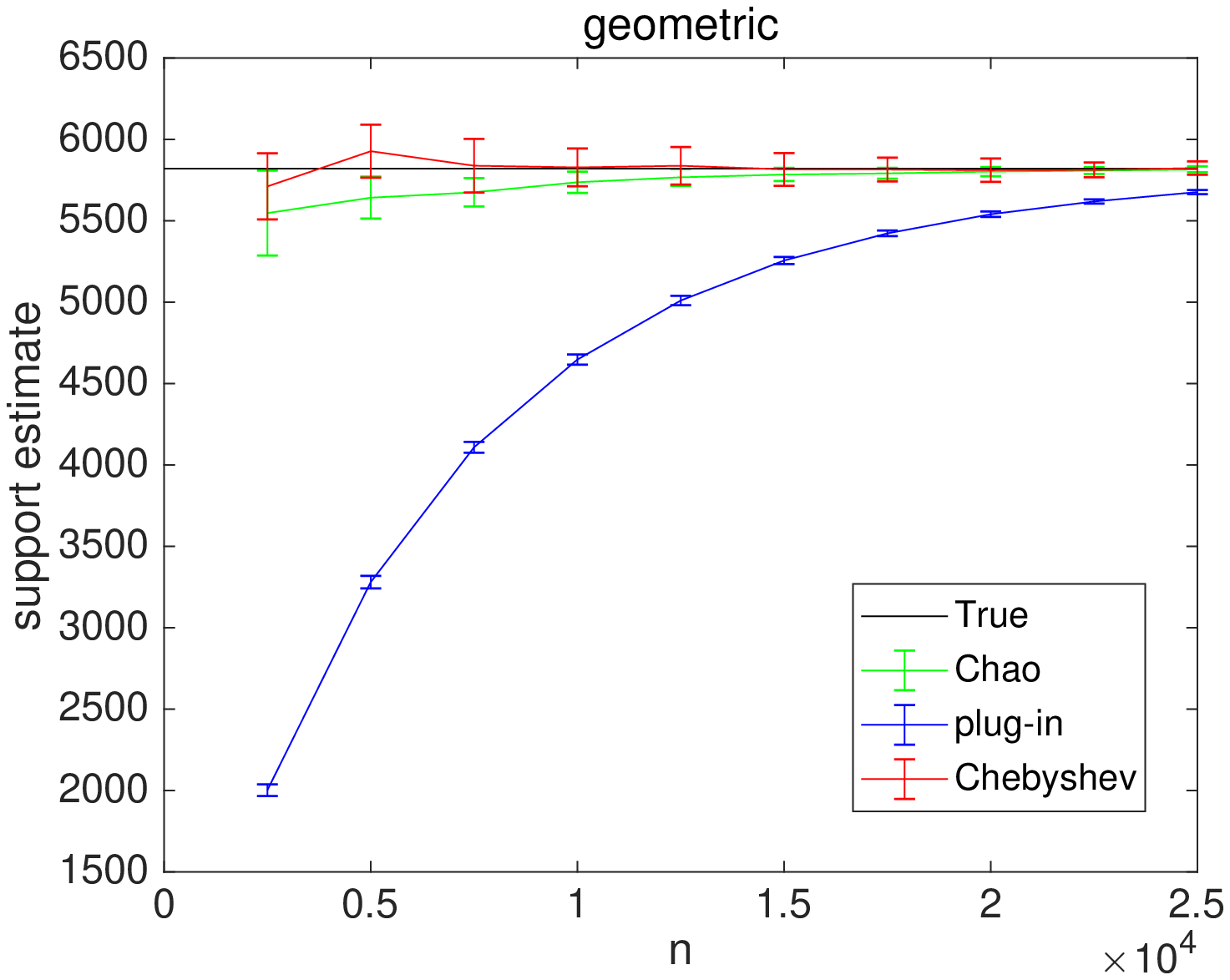}
                \label{fig:geo}
		\end{subfigure}%
        \begin{subfigure}[h]{0.5\textwidth}
                \includegraphics[width=\linewidth]{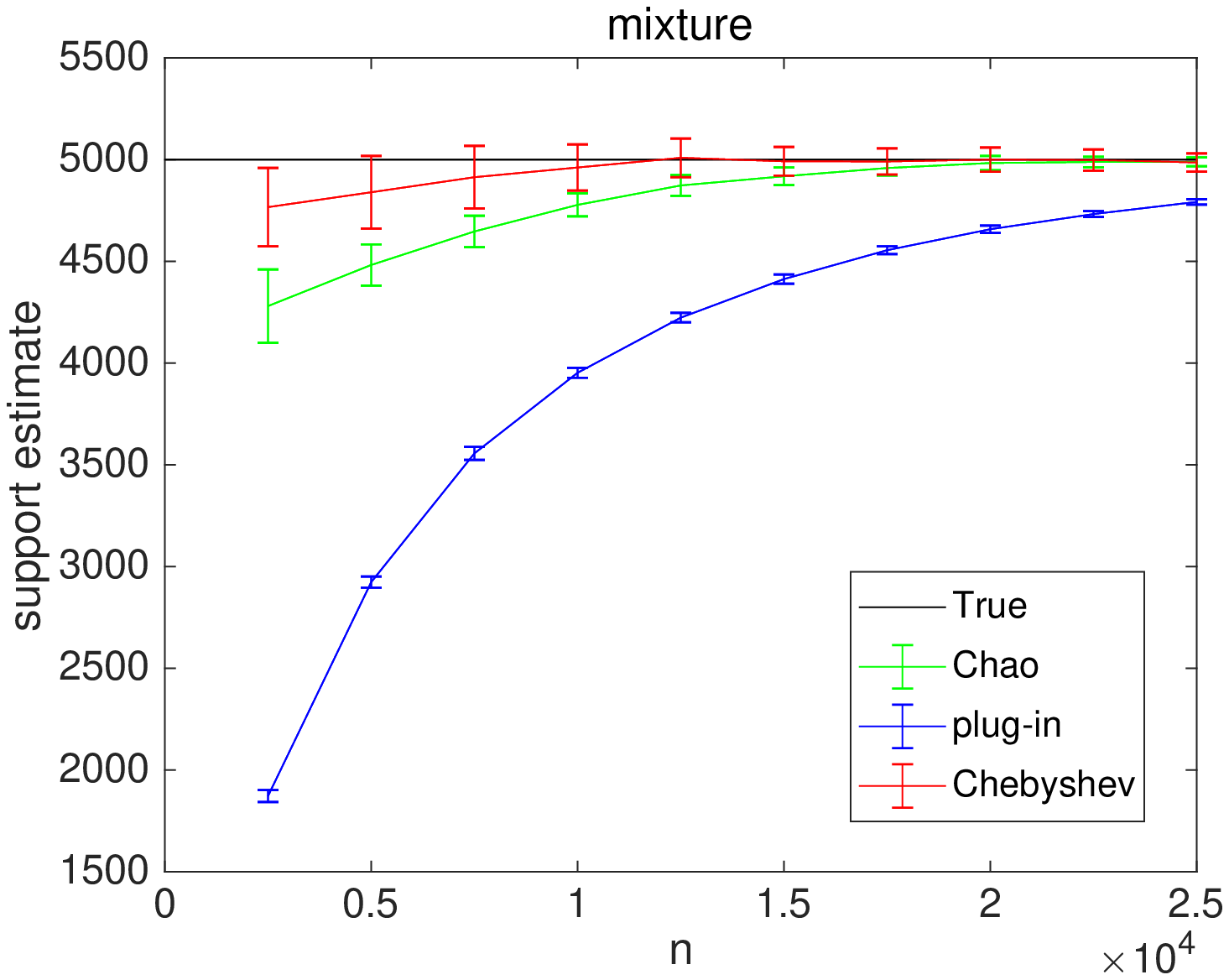}
                \label{fig:unimix}
		\end{subfigure}%
\caption{Comparison of plug-in, Chao and Chebyshev estimators over various distributions.}
\label{fig:1}
\end{figure}

\section{Analysis of the Chao estimator} \label{sec:proof-sketch}
    The MSE of the modified Chao estimator can be written as
    \begin{equation}
    	\cE(\hat{U}^{\text{mc}}, P)=\mathbb{E} \left[ \left(\frac{\varphi_1^2}{2(\varphi_2+1)} \right)^2
              \right] - \mathbb{E} \left[ \frac{\varphi_0
                \varphi_1^2}{\varphi_2+1} \right] + \mathbb{E} \left[ \varphi_0^2
              \right].
              \label{eq:MSE3term}
    \end{equation}
    Analyzing the above quantity is difficult as it involves rational
    functions of prevalences. A natural question to ask is how good are the
    approximations:
    \begin{equation}
    \begin{aligned}
    \label{eq:approx_special1}
    \EE \left[ \left(\frac{\varphi_1^2}{2(\varphi_2+1)} \right)^2\right] &\approx
    \left(\frac{\EE[\varphi_1^2]}{2(\EE[\varphi_2]+1)} \right)^2, \\
    \mathbb{E} \left[ \frac{\varphi_0 \varphi_1^2}{\varphi_2+1} \right] &\approx \frac{\mathbb{E} [\varphi_0] \cdot \mathbb{E}[\varphi_1^2]}{\mathbb{E}[\varphi_2]+1}.
    \end{aligned}
    \end{equation}
    We expect such approximations to hold when $\EE[\varphi_2]$ is
    large. Motivated by this, we divide the proof of
    Theorem~\ref{theorem:main:MSE} into two cases based on $\EE[\varphi_2]$:
    \begin{description}
    	\item[\caseone.] $\bbE [\varphi_2] \ge n^a$, where $a$ is a
              constant that is determined later. In this case, the
              prevalences concentrate around their mean.
    	\item[\casetwo.] $\bbE [\varphi_2] < n^a$. In this case, both the number
              of unseen elements and the estimates are small.
    \end{description}
\subsection{Analysis for \caseone}
\label{sec:rational_approx1}
We first analyze the case where $\EE[\varphi_2]$ is large. Instead of
asking when approximation \eqref{eq:approx_special1} holds, we generalize and ask if
expectations involving such rational functions of prevalences hold. Let
$\varPhi_\poly$ be a homogeneous polynomial of degree $d$ in
$\varphi_i$ and let $\varPhi_\linear$ be a linear function of prevalences of the
form
\[
\varPhi_\linear = \sum_{i \geq 0} \beta_i \varphi_i,
\]
and let 
\begin{equation}
    \label{eq:sigma}
    \sigma \triangleq \beta_0 + \sum_{i \ge 1} \frac{\beta_i}{\sqrt{2 \pi i}}.
\end{equation}
\begin{theorem} \label{theorem:decouple1}
Let $\beta_i \in [0,1]$ for each $i \geq 0$. Then for any
{non-increasing} function $f$,
		\begin{equation} 
		\mathbb{E} \left[ \varPhi_\poly \cdot f(\varPhi_\linear)
                  \right]\ \ge\ \mathbb{E} [\varPhi_\poly] \cdot \mathbb{E} [f(
                  \varPhi_\linear + d)].\label{eq:lower-bound1}
		\end{equation}
If $f$ is \textbf{concave} and $\mathbb{E} [\varPhi_\linear] \ge d
\sigma$,
		\begin{equation} \label{eq:upper-bound1}
		\mathbb{E} \left[ \varPhi_\poly \cdot  f(\varPhi_\linear)
                  \right]\ \le\ \mathbb{E} [\varPhi_\poly] \cdot f(\mathbb{E} [
                  \varPhi_\linear ] - d \sigma ).
		\end{equation}
\end{theorem}
\begin{proof}
A proof is given in Section \ref{sec:rational_approx}.
\end{proof}
Note that if the function $f$ is smooth and has small derivative around $\EE [\varPhi_\linear]$, then Theorem~\ref{theorem:decouple1} implies that 
\[
	\mathbb{E} \left[ \varPhi_\poly \cdot  f(\varPhi_\linear)
                  \right] \approx \mathbb{E} \left[ \varPhi_\poly] \cdot \EE[ f(\varPhi_\linear)
                  \right].
\]
In addition to \eqref{eq:upper-bound1} of Theorem~\ref{theorem:decouple1}, 
which only holds when $f$ is concave, we develop one more such upper bound when $f$ is not concave. This is particularly useful for Chao estimator as the function $f$ in Chao estimator is $1/x$, which is not concave.

Define $\mathbb{V}$ as the space spanned by the functions $\{1,  (x+1)^{-1}, ((x+1)(x+2))^{-1}, \dots\}$ over $\mathbb{R}_{\ge 0}$. Functions in this space are represented as 
$\mathbf{v} = (v_0,v_1,\dots) \equiv \sum_{r \ge 0} v_r \cdot \prod_{j=1}^r (x+j)^{-1}$. A function $f_1$ is said to dominate 
another function $f_2$ over some domain $D$ if $\forall x \in D, f_1 (x) 
\ge f_2 (x)$. Let $\mathrm{Supp} (\varPhi_\linear)$ be the range of function $\varPhi_\linear$. 
\begin{theorem} \label{theorem:decouple:0011}
		Consider $\varPhi_\linear$ with $\beta_i \in [0,1]$ for each $i \geq 0$. Consider some function $f$ and let $(f_0',f_1',\dots) \in \mathbb{V}$ dominate $f$ over $\ \mathrm{Supp} (\varPhi_\linear)$. Then, if $\mathbb{E} [\varPhi_\linear] > d \sigma$,
		\begin{equation*}
		\mathbb{E} \left[ \varPhi_\poly f(\varPhi_\linear) \right] \le \mathbb{E} [\varPhi_\poly] \cdot \sum_{t \ge 0} f'_t (\mathbb{E} [ \varPhi_\linear ] - d \sigma )^{-t}.
		\end{equation*}
	\end{theorem}
\begin{proof}
A proof is given in Section \ref{sec:rational_approx}.
\end{proof}
The above two theorems can be used in other scenarios where expectation of rational functions of prevalences are required, such as computing the expected KL risk for Good-Turing estimators~\cite{OS15} and modified Good-Turing estimators~\cite{acharya2013optimal, hao2019doubly}. 
Using Theorems \ref{theorem:decouple1} and \ref{theorem:decouple:0011}, we approximate $\cE(\hat{U}^{\text{mc}}, P)$ and relate the expectation of ratios (resp. products) to ratio (resp. product) of expectations as required in \eqref{eq:approx_special1}. This results in the following lemma.

\begin{lemma} 
\label{lemma:high:0011}
	For the modified Chao estimator, defining $\sigmaAC = 1/\sqrt{4\pi} = 0.282...$ by \eqref{eq:sigma},
        if $\mathbb{E} [\varphi_2] > 4\sigmaAC$, for any distribution
        $P \in \Delta_k$,
		\begin{equation*}
\cE(\hat{U}^{\text{mc}}, P) \le \left( \frac{\mathbb{E}[\varphi_1^2]}{2
                  \mathbb{E} [\varphi_2]} - \mathbb{E} [\varphi_0]\right)^2
                + \frac{ 4 k^4}{\left( \mathbb{E}
                  [\varphi_2] - 4 \sigmaAC \right)^3}.
\end{equation*}
\end{lemma}
\begin{proof}
We start by upper-bounding the first term of $\cE(\hat{U}^{\text{mc}}, P)$ in \eqref{eq:MSE3term} using Theorem \ref{theorem:decouple:0011}. Let $f(x) = (1+x)^{-2}$. For $x\ge0$, since
\[
\frac{1}{(1+x)^2}\le \frac{1}{(1+x)(2+x)}+\frac{3}{(1+x)(2+x)(3+x)},
\]
we have that $(0,0,1,3,0,\ldots)\in\mathbb{V}$ dominates $f$ in $[0,\infty)$. Setting $\varPhi_\poly=\varphi_1^4$ ($d=4$) and $\varPhi_\linear=\varphi_2$ in Theorem \ref{theorem:decouple:0011}, we get 
\begin{align}
    \mathbb{E} \left[ \left(\frac{\varphi_1^2}{2(\varphi_2+1)} \right)^2
              \right]&\le \frac{\EE[\varphi_1^4]}{4}\left(\frac{1}{(\EE[\varphi_2]-4\sigmaAC)^2}+\frac{3}{(\EE[\varphi_2]-4\sigmaAC)^3}\right).\label{eq:MSE11}
\end{align}
For $0<a<x$ and integer $t\ge 1$, we have $(1-a/x)^t \ge 1-at/x \ge 1-at/(x-a)$. Rearranging, we get
\begin{equation}
    \frac{1}{(x-a)^t}\le \frac{1}{x^t}+\frac{at}{(x-a)^{t+1}},\,\,0<a<x.
\end{equation}
Using the above in the first term of \eqref{eq:MSE11} with $x=\EE[\varphi_2]$, $a=4\sigmaAC$, we get
\begin{align}
    \mathbb{E} \left[ \left(\frac{\varphi_1^2}{2(\varphi_2+1)} \right)^2
              \right]&\le \frac{\EE[\varphi_1^4]}{4\EE[\varphi_2]^2}+\frac{(3+8\sigmaAC)k^4}{4(\EE[\varphi_2]-4\sigmaAC)^3},\label{eq:MSE12}
\end{align}
where we have used $\varphi_1\le k$ in the second term. Using Lemma \ref{lemma:appB:004} with $\varPhi_\poly = \varphi_1^2$, we have
$$\EE[\varphi_1^4]\le \EE[\varphi_1^2]^2+6k\EE[\varphi_1^2]\le \EE[\varphi_1^2]^2+6k^3.$$
Using the above in \eqref{eq:MSE12}, we get
\begin{align}
\mathbb{E} \left[ \left(\frac{\varphi_1^2}{2(\varphi_2+1)} \right)^2
              \right] &\le \frac{\EE[\varphi_1^2]^2}{4\EE[\varphi_2]^2}
              +\frac{6k^3}{4\EE[\varphi_2]^2}
              +\frac{(3+8\sigmaAC)k^4}{4(\EE[\varphi_2]-4\sigmaAC)^3}\nonumber\\
              &\stackrel{(i)}{\le}\frac{\EE[\varphi_1^2]^2}{4\EE[\varphi_2]^2}
              +\frac{6k^3}{4(\EE[\varphi_2]-4\sigmaAC)^2}\,\frac{k}{\EE[\varphi_2]-4\sigmaAC}
              +\frac{(3+8\sigmaAC)k^4}{4(\EE[\varphi_2]-4\sigmaAC)^3}\nonumber\\
              &\le \frac{\EE[\varphi_1^2]^2}{4\EE[\varphi_2]^2}+\frac{(9+8\sigmaAC)k^4}{4(\EE[\varphi_2]-4\sigmaAC)^3},
              \label{eq:MSEterm1}
\end{align}
where $(i)$ follows because $\EE[\varphi_2]-4\sigmaAC\le \EE[\varphi_2] \le k$.

Next, we lower bound the second term of $\cE(\hat{U}^{\text{mc}},P)$ in \eqref{eq:MSE3term}. In \eqref{eq:lower-bound1}, setting $\varPhi_\poly=\varphi_0\varphi_1^2$ ($d=3$), $\varPhi_\linear=\varphi_2$ and $f(x)=(1+x)^{-1}$, we get 
\begin{align}
    \mathbb{E} \left[ \frac{\varphi_0
                \varphi_1^2}{\varphi_2+1} \right] &\ge \EE[\varphi_0
                \varphi_1^2]\,\EE\left[\frac{1}{\varphi_2+4}\right]\nonumber\\
&\stackrel{(i)}{\ge} \EE[\varphi_0\varphi_1^2]\,\frac{1}{\EE[\varphi_2]+4}\nonumber\\
&\stackrel{(ii)}{\ge} \EE[\varphi_0\varphi_1^2]\left(\frac{1}{\EE[\varphi_2]}-\frac{4}{\EE[\varphi_2]^2}\right),\label{eq:MSEterm2}
\end{align}
where $(i)$ follows by Jensen's inequality and $(ii)$ by the inequality $1/(x+a)\ge 1/x-a/x^2$ for $x+a\ge0$.
 
To upper bound the third term of $\cE(\hat{U}^{\text{mc}},P)$ in \eqref{eq:MSE3term}, we use Lemma \ref{lemma:appB:002} (with $h=2$) to get 
\begin{equation}
 \EE[\varphi_0^2]\le \EE[\varphi_0]^2+\EE[\varphi_0].\label{eq:MSEterm3}
\end{equation}
Adding \eqref{eq:MSEterm1}, \eqref{eq:MSEterm2} and \eqref{eq:MSEterm3}, we get
\begin{align}
\cE(\hat{U}^{\text{mc}},P)
&\le \frac{\EE[\varphi_1^2]^2}{4\EE[\varphi_2]^2}+\frac{(9+8\sigmaAC)k^4}{4(\EE[\varphi_2]-4\sigmaAC)^3}
-\frac{\EE[\varphi_1^2\varphi_0]}{\EE[\varphi_2]}+\frac{4\EE[\varphi_1^2\varphi_0]}{\EE[\varphi_2]^2}
+\EE[\varphi_0]^2+\EE[\varphi_0].\label{eq:MSE:ub1}
\end{align}
In \eqref{eq:upper-bound1} of Theorem \ref{theorem:decouple1}, setting $\varPhi_\poly=\varphi_1^2$ ($d=2$), $\varPhi_\linear=\varphi_0$ and $f(x)=-x$, we get 
\begin{equation}
    \EE[\varphi_1^2\varphi_0]\ge \EE[\varphi_1^2]\EE[\varphi_0-2\sigmaAC]=\EE[\varphi_1^2]\EE[\varphi_0]-2\sigmaAC\EE[\varphi_1^2].
\end{equation}
Using the above in the negative third term in \eqref{eq:MSE:ub1} and rearranging, we get
\begin{equation}
\cE(\hat{U}^{\text{mc}},P)\le 
\left(\frac{\EE[\varphi_1^2]}{2\EE[\varphi_2]}-\EE[\varphi_0]\right)^2
+\frac{(9+8\sigmaAC)k^4}{4(\EE[\varphi_2]-4\sigmaAC)^3}
+\frac{2\sigmaAC\EE[\varphi_1^2]}{\EE[\varphi_2]}+\frac{4\EE[\varphi_1^2\varphi_0]}{\EE[\varphi_2]^2}+\EE[\varphi_0].\label{eq:MSE1}
\end{equation}
To get the statement of the lemma, the last three terms above are bounded and combined into the second term as follows. Using $\varphi_i\le k$ in the numerators of the last three terms in \eqref{eq:MSE1} and replacing $\EE[\varphi_2]$ with the smaller $\EE[\varphi_2]-4\sigmaAC$ in the denominators, we get
\begin{equation}
\cE(\hat{U}^{\text{mc}},P)\le 
\left(\frac{\EE[\varphi_1^2]}{2\EE[\varphi_2]}-\EE[\varphi_0]\right)^2
+\frac{(9+8\sigmaAC)k^4}{4(\EE[\varphi_2]-4\sigmaAC)^3}
+\frac{2\sigmaAC k^2}{\EE[\varphi_2]-4\sigmaAC}+\frac{4k^3}{(\EE[\varphi_2]-4\sigmaAC)^2}+k.\label{eq:MSE2}
\end{equation}
Finally, observe that
\begin{align}
    \EE[\varphi_2]-4\sigmaAC \le \EE[\varphi_2] =  \frac{1}{2}\sum_{x\in\mathcal{X}}(np_x)^2e^{-np_x}\stackrel{(i)}{\le} 2e^{-2}\mathcal{X} \le 2e^{-2}k,
\end{align}
where $(i)$ follows because $x^2e^{-x}\le 4e^{-2}$, $x\ge0$. Using the above in \eqref{eq:MSE2},
\begin{align}
    \cE(\hat{U}^{\text{mc}},P)&\le
    \left(\frac{\EE[\varphi_1^2]}{2\EE[\varphi_2]}-\EE[\varphi_0]\right)^2
+\frac{(9+8\sigmaAC)k^4}{4(\EE[\varphi_2]-4\sigmaAC)^3}
+\frac{2\sigmaAC (4e^{-4})k^4}{(\EE[\varphi_2]-4\sigmaAC)^3}\nonumber\\
&\qquad\qquad\qquad\qquad\qquad\qquad+\frac{4(2e^{-2})k^4}{(\EE[\varphi_2]-4\sigmaAC)^3}+\frac{(8e^{-8})k^4}{(\EE[\varphi_2]-4\sigmaAC)^3}\nonumber\\
&\le \left(\frac{\EE[\varphi_1^2]}{2\EE[\varphi_2]}-\EE[\varphi_0]\right)^2+\frac{(9/4+2\sigmaAC+8e^{-4}\sigmaAC+8e^{-2}+8e^{-8})k^4}{(\EE[\varphi_2]-4\sigmaAC)^3},
\end{align}
which results in the statement of the lemma.
\end{proof}
Thus, to bound the MSE of the Chao estimator, we need to bound $
\frac{\mathbb{E}[\varphi_1^2]}{2 \mathbb{E} [\varphi_2]} - \mathbb{E} [\varphi_0]$.
	\begin{lemma} \label{lemma:high:0021}
	For any $P \in \Delta_k$,
	\begin{equation}
	\frac{-k e^{-n/k}}{(1 + n/(k\alpha))^2} \le
        \frac{\mathbb{E}[\varphi_1^2]}{2 \mathbb{E} [\varphi_2]} -
        \mathbb{E} [\varphi_0] \le \frac{k}{n}, \label{eq:biasbound}
	\end{equation}
	where $\alpha=0.5569...$ solves $u^2=4e^{-2}e^{-u}$. Modifying \eqref{eq:biasbound},
	\begin{equation}
	    \left(\frac{\mathbb{E}[\varphi_1^2]}{2 \mathbb{E} [\varphi_2]} -
        \mathbb{E} [\varphi_0]\right)^2 \le \frac{k^2 e^{-2n/k}}{(1 + n/(k\alpha))^4}+\frac{k^2}{n^2}.\label{eq:biassquarebound}
	\end{equation}
	\end{lemma}
\begin{proof}
We first prove the upper bound. By Lemma \ref{lemma:appB:002} (with $h=2$), 
\begin{align}
    \EE[\varphi_1^2]\le \EE[\varphi_1]^2+\EE[\varphi_1].\label{eq:phi1}
\end{align}
Using \eqref{eq:cauchy}, the first term above is upper-bounded as $\EE[\varphi_1]^2\le 2\EE[\varphi_0]\EE[\varphi_2]$. For the second term, we proceed as follows:
\begin{align}
    \EE[\varphi_1] &= \sum_{x\in\mathcal{X}}e^{-np_x}np_x = 2\sum_{x\in\mathcal{X}}e^{-np_x}\frac{(np_x)^2}{2}\frac{1}{np_x}\nonumber\\
    &\stackrel{(i)}{\le} 2\sum_{x\in\mathcal{X}}e^{-np_x}\frac{(np_x)^2}{2}\frac{k}{n}\nonumber\\
    &\le \frac{k}{n}\,2\EE[\varphi_2],\label{eq:ephi1}
\end{align}
where $(i)$ follows by using $p_x\ge1/k$ in the term $1/np_x$. Using \eqref{eq:cauchy} and \eqref{eq:ephi1} in \eqref{eq:phi1}, we get the upper bound of the lemma.

The lower bound is more involved and we provide the proof now. 
 By Jensen's inequality,
\[
	\frac{\bbE [\varphi^2_1]}{2 \bbE \left[ \varphi_2 \right]} - \bbE \left[ \varphi_0 \right] \geq 	\frac{\bbE [\varphi_1]^2}{2 \bbE \left[ \varphi_2 \right]} - \bbE \left[ \varphi_0 \right].
\]
Hence, it suffices to lower bound the RHS above, or upper bound its negative. For ease of exposition, let $\lambda_x$ denote $np_x$ for symbol $x\in\mathcal{X}$. Recall that
\[
\EE[\varphi_i] = \sum_{x \in \cX} e^{-\lambda_x} \frac{\lambda^i_x}{i!}.
\]
Fixing the size of the alphabet $m\triangleq|\mathcal{X}|$ and letting $\lambda=[\lambda_1,\ldots,\lambda_m]$, we define
\[
B(\lambda) \triangleq\bbE \left[ \varphi_0 \right]-\frac{\bbE [\varphi_1]^2}{2 \bbE \left[ \varphi_2 \right]} =\sum_{i=1}^m e^{-\lambda_i} - \frac{(\sum_{i=1}^m \lambda_i e^{-\lambda_i})^2}{\sum_{i=1}^m \lambda_i^2 e^{-\lambda_i}},
\]
where
\begin{equation}
    \lambda\in\Lambda\triangleq\{v\in\mathbb{R}^m:v_i\ge n/k,\sum_i v_i = n\}.
\end{equation}
We relax the domain of $\lambda$ to 
\begin{equation}
    \Lambda'\triangleq\{v\in\mathbb{R}^m:v_i\ge n/k\}\supseteq \Lambda,
\end{equation}
and consider the following optimization problem:
\begin{equation}
    B^* = \max_{\lambda\in\Lambda'} B(\lambda).
\end{equation}
In the rest of this proof, we will show that $B^*\le k\,e^{-n/k}/(1+n/\alpha k)^2$, which implies the lower bound of the lemma.

Since $B(\lambda)$ is continuously differentiable in $\Lambda'$, any extremum point $\lambda^*=[\lambda^*_1,\ldots,\lambda^*_m]\in\Lambda'$ for $B(\lambda)$ satisfies either $\lambda^*_i=n/k$ or $\frac{\partial B}{\partial \lambda_i}\vert_{\lambda^*_i}=0$ for each $i$. Differentiating $B(\lambda)$ partially with respect to $\lambda_i$ and simplifying, we get
\begin{align}
    \frac{\partial B}{\partial \lambda_i} &= -\left(\lambda_i - \frac{2\EE[\varphi_2]}{\EE[\varphi_1]}\right)\left(\lambda_i - 2 - \frac{2\EE[\varphi_2]}{\EE[\varphi_1]}\right)\frac{e^{-\lambda_i}\EE[\varphi_1]^2}{4\EE[\varphi_2]^2}\label{eq:Bp1}\\
    &= -\big(a_{\tsim i}\lambda_i-b_{\tsim i}\big)\big((a_{\tsim i}-2e^{-\lambda_i})\lambda_i-(2a_{\tsim i}+b_{\tsim i})\big)\frac{e^{-\lambda_i}}{4\EE[\varphi_2]^2},\label{eq:Bp2} 
\end{align}
where $a_{\tsim i}=\sum\limits_{i'=1,i'\ne i}^m\lambda_{i'}e^{-\lambda_{i'}}$ and $b_{\tsim i}=\sum\limits_{i'=1,i'\ne i}^m\lambda^2_{i'}e^{-\lambda_{i'}}$. Hence, if 
$\frac{\partial B}{\partial \lambda_i}=0$, then either
\begin{align}
    \lambda_i &= \frac{2\EE[\varphi_2]}{\EE[\varphi_1]} = \frac{b_{\tsim i}}{a_{\tsim i}},\text{ or}\\
    \lambda_i& = 2+\frac{2\EE[\varphi_2]}{\EE[\varphi_1]}\text{ solves }(a_{\tsim i}-2e^{-\lambda_i})\lambda_i=2a_{\tsim i}+b_{\tsim i}.\label{eq:li}
\end{align}
Since $(a_{\tsim i}-2e^{-x})x$ is one-to-one from $[\max(0,\log(2/a_{\tsim i})),\infty)$ to $[0,\infty)$, a solution for $\lambda_i$ exists in \eqref{eq:li}. Also, by \eqref{eq:ephi1}, $\frac{2\EE[\varphi_2]}{\EE[\varphi_1]}\ge n/k$. 

Hence, any extremum point $\lambda^*=[\lambda^*_1,\ldots,\lambda^*_m]\in\Lambda'$ of $B(\lambda)$ necessarily has the following form:
\begin{align}
    \exists S_0\subseteq [m]: &\,\,\lambda^*_i=n/k, \qquad\quad i\in S_0,\\
    \exists S_1\subseteq [m]\setminus S_0: &\,\,\lambda^*_i=\lambda_c, \qquad\qquad i\in S_1,\\
    S_2 = [m]\setminus S_0\setminus S_1: &\,\,\lambda^*_i=2+\lambda_c, \qquad i\in S_2,\\
    &\,\,\lambda_c = \frac{\sum_{i=1}^m(\lambda^*_i)^2 e^{-\lambda^*_i}}{\sum_{i=1}^m\lambda^*_i e^{-\lambda^*_i}}.\label{eq:lc}
\end{align}
Letting $s_0 = |S_0|$, $s_1 = |S_1|$, \eqref{eq:lc} can be written as
\begin{align}
    \lambda_c = \frac{s_0(n/k)^2 e^{-n/k} + s_1\lambda^2_c e^{-\lambda_c} + (m-s_0-s_1)(2+\lambda_c)^2 e^{-(2+\lambda_c)}}{s_0(n/k)e^{-n/k}+s_1\lambda_c e^{-\lambda_c}+(m-s_0-s_1)(2+\lambda_c)e^{-(2+\lambda_c)}}.
\end{align}
Cross-multiplying and simplifying, we get
\begin{align}
    s_0(\lambda_c-n/k)(n/k)e^{-n/k}=2(m-s_0-s_1)(2+\lambda_c)e^{-(2+\lambda_c)}.\label{eq:lc2}
\end{align}
For $s_0=m$, we get $B(\lambda^*)=0$. For $s_0=0$, from \eqref{eq:lc2}, we need $s_1=m$, which results in $B(\lambda^*)=0$. For $s_0+s_1=m$, we have $s_2=0$ and, from \eqref{eq:lc2}, we need $\lambda_c=n/k$, which results in $B(\lambda^*)=0$. For $s_0>0, s_1\ge0$ such that $s_0+s_1<m$, it is easy to see that there exists a unique $\lambda_c>n/k$ satisfying \eqref{eq:lc2} (the LHS increases from 0 at $\lambda_c=n/k$ linearly, and the RHS falls from a non-zero value). For such extremum points, we get $B(\lambda^*)>0$. Therefore,
\begin{align}
    B^* &= \max_{\lambda\in\Lambda'} B(\lambda)\nonumber\\ 
    &= \max_{\stackrel{s_0\in[1:m-1],s_0+s_1<m}{\lambda^*:\lambda_c\text{ solves }\eqref{eq:lc2}}} \left[s_0e^{-n/k}+s_1e^{-\lambda_c}+(m-s_0-s_1)e^{-(2+\lambda_c)} - \frac{\left(\sum_{i=1}^m\lambda^*_i e^{-\lambda^*_i}\right)^2}{\sum_{i=1}^m(\lambda^*_i)^2 e^{-\lambda^*_i}}\right]\nonumber\\
    &\stackrel{(a)}{=}\max_{\stackrel{s_0\in[1:m-1],s_0+s_1<m}{\lambda_c\text{ solves }\eqref{eq:lc2}}}\left[s_0e^{-n/k}+s_1e^{-\lambda_c}+(m-s_0-s_1)e^{-(2+\lambda_c)}\right.\nonumber\\ &\left.\qquad\qquad\qquad\qquad - \frac{s_0(n/k)e^{-n/k}+s_1\lambda_c e^{-\lambda_c} + (m-s_0-s_1)(2+\lambda_c)e^{-(2+\lambda_c)}}{\lambda_c}\right]\nonumber\\
    &=\max_{\stackrel{s_0\in[1:m-1],s_0+s_1<m}{\lambda_c\text{ solves }\eqref{eq:lc2}}} \frac{1}{\lambda_c}\left[s_0(\lambda_c-n/k)e^{-n/k}-2(m-s_0-s_1)e^{-(2+\lambda_c)}\right]\nonumber\\
    &\stackrel{(b)}{=}\max_{\stackrel{s_0\in[1:m-1],s_0+s_1<m}{\lambda_c\text{ solves }\eqref{eq:lc2}}} \frac{2k}{n}(m-s_0-s_1)\frac{1}{\lambda_c}(\lambda_c+2-n/k)e^{-(2+\lambda_c)},\label{eq:Bstar}
\end{align}
where $(a)$ uses \eqref{eq:lc} and $(b)$ follows from \eqref{eq:lc2}.

For a given $s_0$, $s_1$, let 
\begin{align}
    \lambda_c(s_0,s_1)&\triangleq \lambda_c\text{ that solves }s_0(\lambda_c-n/k)(n/k)e^{-n/k}=2(m-s_0-s_1)(2+\lambda_c)e^{-(2+\lambda_c)},\\
    B(s_0,s_1)&\triangleq \frac{2k}{n}(m-s_0-s_1)\frac{1}{\lambda_c(s_0,s_1)}(\lambda_c(s_0,s_1)+2-n/k)e^{-(2+\lambda_c(s_0,s_1))}.
\end{align}
Now, whenever $s_0+s_1+1<m$, it is easy to see that
\begin{equation}
    \lambda_c(s_0,s_1+1) < \lambda_c(s_0,s_1).
\end{equation}
Further, we claim that $B(s_0,s_1)\ge B(s_0,s_1+1)$. To reduce clutter, we denote $\lambda_c\triangleq \lambda_c(s_0,s_1)$ and $\lambda'_c\triangleq \lambda_c(s_0,s_1+1)$. The claim $B(s_0,s_1)\ge B(s_0,s_1+1)$ reduces as follows:
\begin{align}
    \frac{1}{\lambda_c}(\lambda_c+2-n/k)e^{-(2+\lambda_c)}&\ge \frac{m-s_0-s_1-1}{m-s_0-s_1}\frac{1}{\lambda'_c}(\lambda'_c+2-n/k)e^{-(2+\lambda'_c)},\text{ or}\nonumber\\
    \frac{1}{\lambda_c}(\lambda_c+2-n/k)&\ge \frac{\lambda'_c-n/k}{\lambda_c-n/k}\,\frac{\lambda_c+2}{\lambda'_c+2}\,\frac{1}{\lambda'_c}(\lambda'_c+2-n/k),\text{ or}\label{eq:Bs0s1}\\
    \frac{2\lambda'_c+2-n/k}{\lambda'_c(\lambda'_c+2)}&\ge \frac{2\lambda_c+2-n/k}{\lambda_c(\lambda_c+2)},\label{eq:Bs0}
\end{align}
where \eqref{eq:Bs0s1} follows by using \eqref{eq:lc2} for the two cases and \eqref{eq:Bs0} is true because $\frac{2x+2-n/k}{x(x+2)}$ is decreasing for $x>n/k$ and $\lambda'_c < \lambda_c$. Hence, the claim is true.

Now, using $B(s_0,s_1)\ge B(s_0,s_1+1)$ repeatedly, we see that $B(s_0,0)\ge B(s_0,s_1)$ for any $s_1>0$. Hence, in the optimization, it is sufficient to consider $s_1=0$. To reduce clutter, let $\lambda_c(s_0)\triangleq\lambda_c(s_0,0)$ be the solution to
\begin{equation}
    s_0(\lambda_c-n/k)(n/k)e^{-n/k}=2(m-s_0)(2+\lambda_c)e^{-(2+\lambda_c)}
    \label{eq:lambda1}
\end{equation}
in $[n/k,\infty)$. Using the above observations in \eqref{eq:Bstar}, we get
\begin{align}
    B^* &= \max_{s_0\in[1:m-1]} \frac{2k}{n}(m-s_0)\frac{1}{\lambda_c(s_0)}(\lambda_c(s_0)+2-n/k)e^{-(2+\lambda_c(s_0))}\nonumber\\
    &\stackrel{(a)}{=} \max_{\lambda_c(s_0),s_0\in[1:m-1]} \frac{2(\lambda_c(s_0)-n/k)(\lambda_c(s_0)+2-n/k)e^{-(\lambda_c(s_0)+2)}\,m}{\lambda_c(s_0)(2(\lambda_c(s_0)+2)e^{-(\lambda_c(s_0)+2-n/k)}+(n/k)(\lambda_c(s_0)-n/k))}\nonumber\\
    &\stackrel{(b)}{\le} \max_{u\in\mathbb{R},u>n/k} \frac{2(u-n/k)(u+2-n/k)e^{-(u+2)}\,m}{u(2(u+2)e^{-(u+2-n/k)}+(n/k)(u-n/k))},\nonumber\\
    &\stackrel{(c)}{=}2m e^{-n/k}\,\max_{u > 0} \frac{u(u+2)e^{-(u+2)}}{(u+n/k)(2(u+2+n/k)e^{-(u+2)}+u n/k))},\label{eq:BS2}
\end{align}
where $(a)$ follows by using \eqref{eq:lambda1} to write $s_0$ in terms of $\lambda_c(s_0)$, $(b)$ follows by relaxing $\{\lambda_c(s_0):s_0\in[1:m-1]\}$ to a real-valued $u\in[n/k,\infty)$, and $(c)$ follows by the substituting $u$ by $u+n/k$.

The maximization in \eqref{eq:BS2} is relatively straight-forward using calculus and is achieved at 
\begin{equation}
    \alpha=0.5569... \text{ that solves }u^2=4e^{-(u+2)}.
\end{equation}
Replacing $2e^{-(\alpha+2)}$ with $\alpha^2/2$ in the RHS of \eqref{eq:BS2} and simplifying, we get $B^*\le k\,e^{-n/k}/(1+n/(\alpha k))^2$.
\end{proof}
Combining Lemmas \ref{lemma:high:0011} and \ref{lemma:high:0021} and assuming $E[\varphi_2]\ge n^{4/5}$ results in Theorem~\ref{theorem:main:MSE} for \caseone.
            
\subsection{Analysis for \casetwo}
If $\EE[\varphi_2]$ is small, then it is not possible to prove general
results as in Theorem~\ref{theorem:decouple1} and say that the MSE must be 
small. The case $\varPhi_\poly = \varphi_1^2$ and 
$\varPhi_\linear = \varphi_\infty$ illustrates this claim, since 
$\bbE [\varPhi_\linear]$ is always $0$, but the MSE $\approx \bbE [\varPhi_\poly]$
can be as high as $\Theta (k^4)$ for a certain range of $n$.  Hence, modifying our approach for small $\EE[\varphi_2]$, we show that both the Chao estimate and the number of unseen symbols 
are small.
	\begin{lemma} \label{case:0021}
For the Chao estimator, if $\mathbb{E} [\varphi_2] \le n^{4/5}$, for any
distribution $P \in \Delta_k$,
\begin{equation}
\cE(\hat{U}^{\text{mc}}, P) \le \left( \frac{(32.28)k^4}{n^{12/5}} + \frac{(98.97)k^3}{n^{11/5}} + \frac{2k^2}{n^{6/5}} + 
\frac{(1.77)k}{n^{1/5}} + \frac{(21.21)k^2}{n^2} \right).
\end{equation}
\end{lemma}

Recall that in \casetwo we assume that $\bbE [\varphi_2] < n^{4/5}$. Our strategy is to show that when $\bbE [\varphi_2]$ is small, then the unseen elements as well as the estimates are small on average. The idea of negative regression between random variables will play a role in proving Lemma \ref{case:0021}.
%
	
	Negative regression is a strong notion of negative dependence between random variables. It is closely related to negative correlation and negative association \cite{NA-83} between random variables. We begin with its definition,

	\begin{definition} \cite[Definition 21]{dubhashi1996balls} \label{definition:NR:001}
		Let $\mathbf{X} := \{X_1,\dots,X_m\}$ be a set of random variables. $\mathbf{X}$ satisfies the negative regression condition if $\mathbb{E}\left[f(X_i, i \in I) \middle| X_j = t_j , j \in J \right]$ is non-increasing in each $t_j, j \in J$ for any disjoint $I,J \subseteq [m]$ and any non-decreasing (coordinate-wise) function $f$.
	\end{definition}
	
	\noindent To make the connection with negative regression, we first introduce the classical balls and bins experiment. Consider a set of $n$ balls and $m$ bins. In the experiment, each ball is tossed into one of the $m$ bins as per some distribution independent of the others (balls need not have the same distribution of probabilities of going into different bins). There is an intuitive notion of negative dependence in the balls and bins experiment - if a particular bin, say $i$, was revealed to hold a comparatively high number of balls, one would expect the other bins to hold fewer number of balls (because there are fewer balls left to go into the other bins). This notion is formalized in the following theorem.

	\begin{theorem} \cite[Theorem 31]{dubhashi1996balls} \label{theorem:NR:002}
		The set $\mathbf{B} = \{B_1,B_2,\dots\}$, where $B_i$ represents the number of balls in bin $i$ satisfies the negative regression condition.
	\end{theorem}
	
By the negative regression condition, we can conclude that for all bins $j \ne i$, $\bbE [B_j | B_i = \delta]$ is a decreasing function of $\delta$. We are now ready to introduce the lemma that brings these results into context.

	\begin{theorem} \label{theorem:NR:003}
		Under the Poisson sampling model, $\{\varphi_0,\varphi_1,\dots\}$ satisfy the negative regression condition.
	\end{theorem}
	\begin{proof} The heart of the proof lies in the fact that the Poisson arrival process can be thought of as a version of balls and bins where the set $\{\varphi_i, i \in \bbN_0 \}$ plays the role of $\mathbf{B}$, as elaborated below.

    Each symbol $x \in \mathcal{X}$ is a ball. A ball $x$ is in Bin $i$, $i=0,1,2,\ldots$, if $N_x=i$.
		Each ball $x$ is put into bin $i$ with probability $P(N_x = i)$ independently. The total number of balls in bin $i$ is therefore $\sum_{x \in \mathcal{X}} \mathbbm{1}_{N_x = i} = \varphi_i$. The vector $\mathbf{B}$ of Theorem \ref{theorem:NR:002} is, therefore, equivalent to $\{ \varphi_i, i \in \bbN_0 \}$. This concludes the proof.
	\end{proof}

In order to prove Lemma \ref{case:0021}, we first use negative regression to establish an upper bound on the MSE for any distribution in $P \in \Delta_k$ as a function of $\bbE[\varphi_2]$ and show that if $\bbE [ \varphi_2 ]$ is small, this bound is small too.

\begin{lemma} \label{lemma:small:001}
For the modified Chao estimator, for any distribution $P \in \Delta_k$,
\begin{align}
\cE(\hat{U}^{\text{mc}}, P) \le (4+8a) (k/n)^4\bbE \left[ \varphi_2 \right]^2 +  (28a(k/n)^3+2(k/n)^2+0.5a(k/n))\bbE \left[ \varphi_2 \right] + 6a(k/n)^2,
\end{align}
where $a = 1/(1-2e^{-2})^4$.
\end{lemma}
\begin{proof}
From the definition of MSE,
\begin{align}
\cE(\hat{U}^{\text{mc}}, P) &= \mathbb{E} \left[ \left( \frac{\varphi_1^2}{2 (\varphi_2+1)} - \varphi_0 \right)^2 \right] \nonumber\\
&\le \underbrace{\mathbb{E} \left[ \left( \frac{\varphi_1^2}{2 (\varphi_2+1)} \right)^2 \right]}_{(a)} + \underbrace{\mathbb{E} \left[ \varphi_0^2 \right]}_{(b)}.\label{eq:MSEsplit}
\end{align}
In order to upper bound the MSE, we separately upper bound the quantities $(a)$  and $(b)$. By definition of conditional expectation,
\begin{equation}
\mathbb{E} \left[ \left( \frac{\varphi_1^2}{2 (\varphi_2+1)} \right)^2 \right] = \frac14 \sum_{j \in \bbN_0} \frac{\bbE \left[ \varphi_1^4 | \varphi_2 = j \right]}{(j+1)^2} P(\varphi_2 = j).
\end{equation}
Using the negative regression of $\{ \varphi_i,\ i=0,1,2,\dots \}$, we can conclude that $\forall \delta \ne 0,\ \bbE [\varphi_1^4 | \varphi_2 = 0] \ge \bbE [\varphi_1^4 | \varphi_2 = \delta]$. Therefore,
\begin{align}
\mathbb{E} \left[ \left( \frac{\varphi_1^2}{2 (\varphi_2+1)} \right)^2 \right] &\le \frac14\mathbb{E} \left[\varphi_1^4 | \varphi_2 = 0 \right] \cdot \bbE \left[ (\varphi_2 + 1)^{-2} \right] \nonumber\\
&\overset{(i)}{\le} \frac14 \frac{\bbE \left[ \varphi_1^4 \right]}{\left(1-2e^{-2} \right)^4} \cdot \bbE \left[ (\varphi_2 + 1)^{-2} \right],
\end{align}
where $(i)$ uses Lemma \ref{lemma:appB:001} to bound the conditional expectation of $\varphi_1^4$. Using Lemma \ref{lemma:appB:002} (with $h=4$) to bound $\bbE [\varphi_1^4]$,
\begin{align}
\mathbb{E} \left[ \left( \frac{\varphi_1^2}{2 (\varphi_2+1)} \right)^2 \right] &= \left( \frac{\bbE \left[ \varphi_1 \right]^4 + 7\bbE \left[ \varphi_1 \right]^3 + 6\bbE \left[ \varphi_1 \right]^2 + \bbE \left[ \varphi_1 \right]}{4\left(1-2e^{-2} \right)^4} \right) \cdot \bbE \left[ \frac{1}{(\varphi_2 + 1)^{2}} \right] \nonumber\\
&= \left( \frac{\bbE \left[ \varphi_1 \right]^4 + 7\bbE \left[ \varphi_1 \right]^3 + 6\bbE \left[ \varphi_1 \right]^2}{4\left(1-2e^{-2} \right)^4} \right) \cdot \bbE \left[ \frac{1}{(\varphi_2 + 1)^{2}} \right]+ \frac{\bbE \left[ \varphi_1 \right]}{4\left(1-2e^{-2} \right)^4}\cdot \bbE \left[ \frac{1}{(\varphi_2 + 1)^{2}} \right] \nonumber\\
&\le \underbrace{2 \left(\frac{\bbE \left[ \varphi_1 \right]^4 + 7\bbE \left[ \varphi_1 \right]^3 + 6\bbE \left[ \varphi_1 \right]^2}{ 4\left(1-2e^{-2} \right)^4 \bbE \left[ \varphi_2 \right]^2} \right)}_{(i)} +  \underbrace{\vphantom{\frac{\bbE [\varphi_1]^4}{ \bbE [\varphi_2]^2}} \frac{\bbE \left[ \varphi_1 \right]}{4\left(1-2e^{-2} \right)^4}}_{(ii)} \nonumber\\[5pt]
&\overset{(iii)}{\le} \frac{(\frac{2k}{n})^4 \bbE [\varphi_2]^2 + \left( 7 (\frac{2k}{n})^3 + \frac{k}{n} \right)\bbE \left[ \varphi_2 \right] + 6(\frac{2k}{n})^2}{2 \left(1-2e^{-2} \right)^4}, \label{eq:14632}
\end{align}
where $(i)$ follows from the fact that $\bbE [(1 + \varphi_2)^{-2}] \le 2\bbE [\varphi_2]^{-2}$ from Theorem \ref{theorem:decouple:0011} with $f(x)=(1+x)^{-2}$, $f'=(0,0,2,0,\ldots)$, $(ii)$ uses $\bbE [(1+ \varphi_2)^{-2}] \le 1$, and $(iii)$ follows from \eqref{eq:ephi1}. 

We now upper bound $(b)$ in \eqref{eq:MSEsplit}. Using Lemma \ref{lemma:appB:004} ($h=2$, $j=0$),
\begin{align}
\bbE[\varphi_0^2] &\le \bbE \left[ \varphi_0 \right]^2 + \bbE \left[ \varphi_0 \right], \nonumber\\
&\overset{(a)}{\le} 4(k/n)^4  \bbE \left[ \varphi_2 \right]^2 + 2(k/n)^2  \bbE \left[ \varphi_2 \right], \label{eq:12930}
\end{align}
where $(a)$ follows because $\bbE[\varphi_0]\le 2(k/n)^2 \bbE[\varphi_2]$, which is proved as follows.
\begin{align}
    \EE[\varphi_0] &= \sum_{x\in\mathcal{X}}e^{-np_x} = 2\sum_{x\in\mathcal{X}}e^{-np_x}\frac{(np_x)^2}{2}\frac{1}{(np_x)^2}\nonumber\\
    &\stackrel{(i)}{\le} 2\sum_{x\in\mathcal{X}}e^{-np_x}\frac{(np_x)^2}{2}(k/n)^2\nonumber\\
    &\le 2(k/n)^2\EE[\varphi_2],\label{eq:ephi02}
\end{align}
where $(i)$ follows using $p_x>1/k$.

Combining \eqref{eq:14632} and \eqref{eq:12930}, proves the lemma.
\end{proof}
\begin{proof}[Proof of Lemma~\ref{case:0021}]
Lemma \ref{case:0021} follows by substituting the upper bound $\bbE [\varphi_2] \le n^{4/5}$ into Lemma \ref{lemma:small:001}.
\end{proof}

Combining Lemmas~\ref{lemma:high:0011},~\ref{lemma:high:0021} and~\ref{case:0021}, the proof of Theorem~\ref{theorem:main:MSE} is complete for all cases. In the rest of the paper, we provide detailed proofs for Theorems~\ref{theorem:decouple1} and~\ref{theorem:decouple:0011}.

\section{Analysis of rational estimators}
We use the notation $\bbN$ to denote the set of natural numbers $\{1,2,\dots \}$ and $\bbN_0$ to denote the set of whole numbers, $\bbN \cup \{0\}$. In addition, for some $i \in \bbN$ we use $[i]$ to denote the set $\{1,\dots,i \}$. 
\label{sec:rational_approx}
The degree-$d$, homogeneous $\varPhi_\poly$ can be expressed as
\begin{equation}
  \varPhi_\poly=\sum_{i^d}\alpha_{i^d}\varphi_{i_1}\cdots\varphi_{i_d}=\sum_{i^d}\alpha_{i^d}\sum_{x^d\in\mathcal{X}^d}\prod_{t=1}^d\mathbbm{1}_{N_{x_t} = i_t},\label{eq:O1}
\end{equation}
where $i^d=[i_1,\ldots,i_d]$, $i_j\in\{0,1,\ldots\}$, $x^d=[x_1,\ldots,x_d]$, $x_j\in\mathcal{X}$, $\alpha_{i^d}\in\mathbb{R}$. Recall that 
$$\varPhi_{\linear}=\sum_{i\ge0}\sum_{u\in\mathcal{X}}\beta_i\mathbbm{1}_{N_u = i}.$$
Note that both $\varPhi_\poly$ and $\varPhi_\linear$ are functions of $N_x$, $x\in\mathcal{X}$. To proceed further, we make the following two definitions.
\begin{align}
    S(i^d,x^d)&\triangleq\prod_{t=1}^d\mathbbm{1}_{N_{x_t} = i_t},\\
    T(i^d,x^d)&\triangleq\varPhi_{\linear}|_{N_{x_1}=i_1,\ldots,N_{x_d}=i_d}.
\end{align}
In words, $T(i^d,x^d)$ denotes evaluation of $\varPhi_\linear$ by setting $N_{x_j} = i_j$, $j=1,\ldots,d$. Now,
\begin{align}
\varPhi_\poly f(\varPhi_{\linear})&=\sum_{i^d}\alpha_{i^d}\sum_{x^d}\left[S(i^d,x^d) f(\varPhi_{\linear})\right]\nonumber\\
&\stackrel{(i)}{=}\sum_{i^d,x^d}\alpha_{i^d}S(i^d,x^d) f(T(i^d,x^d)),\label{eq:O21}
\end{align}
where $(i)$ follows because $S(i^d,x^d)=1$ only when $N_{x_j} = i_j$, $j=1,\ldots,d$. Note that $S(i^d,x^d)$ and $T(i^d,x^d)$ do not involve any common $N_x$ terms and are independent. Taking expectations in \eqref{eq:O21} and using the independence,
\begin{equation}
    \EE[\varPhi_\poly f(\varPhi_{\linear})]=\sum_{i^d,x^d}\alpha_{i^d}\EE[S(i^d,x^d)]\EE[f(T(i^d,x^d))].
    \label{eq:start1}
\end{equation}
The above equality is the main starting point for the proofs.

\subsection{Proof of Theorem \ref{theorem:decouple1}}
For a given $i^d$ and $x^d$, we can write $T(i^d,x^d)$ as follows:
\begin{align} 
    T(i^d,x^d)&=\sum_{j=1}^d \beta_{i_j}+\sum_{u\in\mathcal{X}\setminus x^d}\sum_{i\ge0}\beta_i\mathbbm{1}_{N_u=i}.
\end{align}
To compare the above with $\varPhi_\linear$, we rewrite $\varPhi_\linear$ as follows:
\begin{align}
    \varPhi_\linear&=\sum_{j=1}^d \sum_{i\ge 0}\beta_i \mathbbm{1}_{N_{x_j}=i}+\sum_{u\in\mathcal{X}\setminus x^d}\sum_{i\ge0}\beta_i\mathbbm{1}_{N_u=i}.
\end{align}
Hence, we see that
\begin{align}
    T(i^d,x^d)=\varPhi_\linear-\sum_{j=1}^d \sum_{i\ge 0}\beta_i \mathbbm{1}_{N_{x_j}=i}+\sum_{j=1}^d \beta_{i_j}.\label{eq:Tphi}
\end{align}
Since $\beta_i\le1$, from \eqref{eq:Tphi}, we have $T(i^d,x^d)\le \varPhi_\linear+d$. Since $f$ is non-increasing, we have
\begin{equation}
    \EE[f(T(i^d,x^d))]\ge \EE[f(\varPhi_\linear+d)].
\end{equation}
Using the above in \eqref{eq:start1}, we get the lower bound in \eqref{eq:lower-bound1}.

Since $f$ is concave, by Jensen's inequality,
\begin{equation}
    \EE[f(T(i^d,x^d))]\le f(\EE[T(i^d,x^d)]).\label{eq:fJ}
\end{equation}
In \eqref{eq:Tphi}, dropping the third positive term on the RHS and taking expectations, we get
\begin{align}
    \EE[T(i^d,x^d)]&\ge \EE[\varPhi_\linear]-\sum_{j=1}^d \sum_{i\ge 0}\beta_i P(N_{x_j}=i)\nonumber\\
    &\stackrel{(i)}{\ge} \EE[\varPhi_\linear]-\sum_{j=1}^d \sum_{i\ge 0}\beta_i/\sqrt{2\pi i}\nonumber\\
    &=\EE[\varPhi_\linear]-d\sigma,\label{eq:Tlb}
\end{align}
where $(i)$ follows because $P(N_u=j)=j^je^{-j}/j!\le 1/\sqrt{2\pi j}$ by Stirling's approximation for $u\in\mathcal{X}$ and $j\ge1$. Using the above in \eqref{eq:fJ}, since $f$ is non-increasing, we get
\begin{align}
    \EE[f(T(i^d,x^d))]\le f(\EE[\varPhi_\linear]-d\sigma).
\end{align}
Using the above in \eqref{eq:start1}, we get the upper bound in \eqref{eq:upper-bound1}.

  \subsection{Proof of Theorem~\ref{theorem:decouple:0011}}
	If $f$ is not concave, arriving at upper bounds as in Theorem
    \ref{theorem:decouple1} is less straightforward. However, they are necessary
    to analyze Chao estimator as the function $(1+\varphi_2)^{-2}$ in \eqref{eq:approx_special1} is not concave. An additional property of $\varPhi_\linear$ is 
    required to arrive at upper bounds on the approximation error for such
    functions.

Observe that $\varPhi_\linear = \sum_{i \ge 0} \beta_i \varphi_i$
          can be expanded as
	\begin{equation}
		\varPhi_\linear = \sum_{x \in \mathcal{X}} \sum_i \beta_i \mathbbm{1}_{N_x = i} = \sum_{x \in \mathcal{X}} Y_x,
	\end{equation}
    where each $Y_x$ is a discrete random variable that takes value
    $\beta_i$ with probability $P(N_x = i)$. The restriction of
    $\beta_i \in [0,1]$ for Theorem \ref{theorem:decouple:0011} implies that $\mathrm{Supp} (Y_x) \subseteq
    [0,1]$. In the Poisson sampling model, the random variables $Y_x$, $x\in\mathcal{X}$,
    are independent. Hence, $\varPhi_\linear$ is the sum of independent discrete random variables each supported on some subset of $[0,1]$. We term such random variables as
    \textit{generalized Poisson binomial} random variables. 

	The crucial result is the following. 
	\begin{lemma}
	If $X=\sum_i X_i$ with $X_i\in[0,1]$ being independent discrete random variables, 
	\begin{equation}
	    \int_0^u\EE[z^X]dz\le \EE[X]^{-1} \EE[u^X],\,u\in(0,1].
	\end{equation}
	\label{lem:zX}
	\end{lemma}
	\begin{proof}
	See Section \ref{sec:lemzXproof} for a proof.
	\end{proof}
	Let $f_r (X) = \prod_{j=1}^{r} (X+j)^{-1}$, $f_0(X)=1$. It is easy to see that
	\begin{equation}
	    \EE[f_r(X)] = \left[\int\limits_0^{u_0}\int\limits_0^{u_1}\cdot\cdot\int\limits_0^{u_{r-1}} \EE[u_r^X]du_r\cdots du_2du_1\right]_{u_0=1}.
	\end{equation}
    Using Lemma \ref{lem:zX} $r$ times, we get
	\begin{equation} \label{eq:inv1}
		\bbE [f_r (X)] \le \bbE [X]^{-r}.
	\end{equation}
	Note that $T(i^d,x^d)$ is a generalized Poisson binomial random variable and satisfies Lemma \ref{lem:zX}. 
	
	Now, we are ready to prove the theorem. From the hypothesis of the theorem, $(f'_0,f'_1,\ldots)\in\mathbb{V}$ dominates $f$. Thus we have
	\begin{align}
	    \EE[f(T(i^d,x^d))]&\le \sum_{i\ge0} f'_t\, \EE[f_t(T(i^d,x^d))]\nonumber\\
	    &\stackrel{(i)}{\le} \sum_{t\ge0} f'_t\, \EE[T(i^d,x^d)]^{-t}\nonumber\\
	    &\stackrel{(ii)}{\le} \sum_{t\ge0} f'_t\, (\EE[\varPhi_\linear]-d\sigma)^{-t},
	\end{align}
where $(i)$ uses \eqref{eq:inv1} and $(ii)$ uses \eqref{eq:Tlb}. Using the above in \eqref{eq:start1} concludes the proof of the theorem.
	\printbibliography
	
	\appendix 
\section{Proof of Lemma \ref{lem:zX}}
\label{sec:lemzXproof}

For any discrete random variable $X$, the \textit{characteristic polynomial} $C_X(z) : \mathbb{R}_{> 0} \to \mathbb{R}$ is equal to $\mathbb{E} \left[ z^X \right]$ wherever this expectation exists. 

To prove Lemma \ref{lem:zX}, we show that the characteristic polynomial $C_x(z)$ satisfies for $y \in (0,1]$:
        \begin{equation}
            C_X (y) \le \mathbb{E} [X]^{-1} (D C_X) (y),
        \end{equation}
		where $D$ denotes the differentiation operator $f (t) \mapsto \frac{\diff f}{\diff t} (t)$. Integrating both sides from $y = 0$ to $u$ completes the proof of Lemma \ref{lem:zX}. 
		
Consider a generalized Poisson binomial random variable $X = \sum_i^m X_i$, where each $X_i$ is supported on some $\{d_{ij} \} \subseteq [0,1]$. From the definition of characteristic function, we write
\begin{equation} \label{eq:123rtytuy}
	C_X (z) = \bbE [z^{X}] = \bbE [z^{\sum_i X_i}]  \overset{(i)}{=} \prod_{i=1}^m C_{X_i} (z) = \prod_{i=1}^m \left( \sum_j P (X_i = d_{ij})\ z^{d_{ij}} \right),
\end{equation}			
where $(i)$ follows from the independence of the $X_i$'s. Differentiating both sides of \eqref{eq:123rtytuy}, it follows that,
\begin{align*}
	(D C_X) (z) &= \sum_{i=1}^m \left( \sum_{j : d_{ij} \ne 0} P(X_i = d_{ij}) \ d_{ij} \ z^{d_{ij} - 1} \right) \prod_{k \ne i} \left( \sum_{j'} P(X_k = d_{kj'})\ z^{d_{kj'}} \right), \\
	&\overset{(i)}{\ge} \sum_{i=1}^m \left( \sum_{j : d_{ij} \ne 0} P(X_i = d_{ij}) \ d_{ij} \ z^{d_{ij} - 1} \right) \prod_{k} \left( \sum_{j'} P(X_k = d_{kj'})\ z^{d_{kj'}} \right), \\
	&=  \sum_{i=1}^m \left( \sum_{j : d_{ij} \ne 0} P(X_i = d_{ij}) \ d_{ij} \ z^{d_{ij} - 1} \right) C_X (z) ,\\
	&\overset{(ii)}{\ge} C_X(z) \sum_{i=1}^m \sum_{j} P (X_i = d_{ij}) \ d_{ij} = C_X(z) \, \mathbb{E} [X],
\end{align*}
where $(i)$ follows because $\sum_{j : d_{ij} \ne 0} P(X_i = d_{ij})\ z^{d_{ij}} \le 1$ for $z \in (0,1]$, and $(ii)$ follows because $z^{d_{ij}-1} \ge 1$ for all $i$ and $j$, $d_{ij} \in [0,1]$.

\section{Inequalities for moments of prevalences}
\label{sec:auxlemmas}
\begin{lemma} \label{lemma:appB:002}
For all $j \ge 0$ and $h \ge 1$,
\begin{equation}
\mathbb{E} [\varphi_j^h] \le \sum_{k=1}^h c_{h,k} \mathbb{E}[\varphi_j]^k \quad \text{ where } c_{h,1} = 1 \text{ and } c_{h,k} = \sum_{l=k-1}^{h-1}\binom{h-1}{l}c_{l,k-1}.\label{eq:lemEphi}
\end{equation}
\end{lemma}
\begin{proof}
Since $\varphi_j = \sum_x \mathbbm{1}_{N_x = j}$ is the sum of independent Bernoulli random variables, it has moment generating function, $M(t) = \prod_{x \in \mathcal{X}} (1-P(N_x = j) + P (N_x = j) e^t)$.
Let us define the function $E(t,x) \overset{\mathrm{def}}{=} P (N_x = j) e^t$ and $Z (t,S) \overset{\mathrm{def}}{=} \prod_{\substack{x' \in S}} (1 - P(N_{x'} = j) + P(N_{x'} = j) e^t)$. Note that $Z (t,\mathcal{X}) = M(t)$ and $E(t)$ satisfies the property $(DE) (t) = E(t)$ where $D$ denotes the differentiation operator, $f (t) \mapsto \frac{\diff}{\diff t} f(t)$. Then, for all $t \in [0,1]$ and a nonempty $S \subseteq \mathcal{X}$, differentiating $Z(t,S)$ results in,
\begin{align}
(D Z) (t,S) &= \sum_{x \in S} P (N_x = j) e^t \prod_{\substack{x' \in S \\ x' \ne x}} (1 - P(N_{x'} = j) + P(N_{x'} = j) e^t) \nonumber\\
&=\sum_{x \in S} E (t,x) \cdot Z (t, S \setminus \{ x \}). \label{eq:Zsat}
\end{align}
For functions $f$ and $g$, by the general Leibniz rule,
\begin{equation} \label{eq:012938}
	D^n (fg) = \sum_{k = 0}^n \binom{n}{k} (D^{n-k} f) (D^k g).
\end{equation}
Using \eqref{eq:012938} with $f = Z(t,S \setminus \{x\})$ and $g = E(t,x)$, taking $n=m-1$ and using the fact that $(D E) (t) = E(t)$,
\begin{equation*}
	D^{m-1} \left( E (t,x) \cdot Z(t, S \setminus \{x \}) \right) = E (t,x) \sum_{k=0}^{m-1} \binom{m-1}{k} D^{k} Z (t,S \setminus \{ x \}).
\end{equation*}
Summing both sides over $x \in S$ and using \eqref{eq:Zsat},
\begin{equation} \label{eq:123546}
	D^m Z(t,S) = \sum_{x \in S} E(t,x) \sum_{k=0}^{m-1} \binom{m-1}{k} D^k Z(t,S \setminus \{x \})
\end{equation}
Now, observe that since $1 - P(N_x = j) + P(N_x = j) e^t \ge 1$ for $t \ge 0$, we have that $Z(t,S_1) \le Z (t,S_2)$ if $S_1 \subseteq S_2$. Therefore, with this as the base case, an inductive argument using \eqref{eq:123546} shows that $D^m Z(t,S_1) \le D^m Z(t,S_2)$ if $S_1 \subseteq S_2$. Therefore, we may upper-bound \eqref{eq:123546} by replacing each $D^k Z(t, S \setminus \{ x\})$ by $D^k Z(t,S)$ resulting in
\begin{equation}
    D^m Z(t,S) \le \sum_{x \in S} E (t,x) \sum_{k=0}^{m-1} \binom{m-1}{k} D^k Z(t,S) \label{eq:019}
\end{equation}
Observe from its definition that $Z (t,\mathcal{X}) = M(t)$. Therefore, from \eqref{eq:019} with $t=0$, $m=h$ and $S = \mathcal{X}$,
\begin{align} 
\mathbb{E} [\varphi_j^h] &\le \mathbb{E} [\varphi_j] \, \sum_{k=0}^{h-1} \binom{h-1}{k} \mathbb{E} [\varphi_j^{k}]\nonumber\\
&=\mathbb{E} [\varphi_j] \left[ 1+ \sum_{l=1}^{h-1} \binom{h-1}{l} \mathbb{E} [\varphi_j^{l}]\right].\label{eq:123sdfkj}
\end{align}
For $h=1$, \eqref{eq:lemEphi} is trivially true. For $h=2$, \eqref{eq:lemEphi} is proved by \eqref{eq:123sdfkj}. Now, as an induction hypothesis, suppose that \eqref{eq:lemEphi} holds up to and including $h-1$. Using the induction hypothesis in \eqref{eq:123sdfkj}, we get
\begin{align}
    \mathbb{E} [\varphi_j^h] &\le \mathbb{E} [\varphi_j] \left[ 1+ \sum_{l=1}^{h-1} \binom{h-1}{l} \sum_{k=1}^l c_{l,k}\mathbb{E}[\varphi_j]^k\right]\nonumber\\
    &\stackrel{(a)}{=}\mathbb{E} [\varphi_j] + \sum_{k=1}^{h-1} \left(\sum_{l=k}^{h-1} \binom{h-1}{l}  c_{l,k}\right)\mathbb{E}[\varphi_j]^{k+1}\nonumber\\
    &=\mathbb{E} [\varphi_j] + \sum_{k=2}^{h} \left(\sum_{l=k-1}^{h-1} \binom{h-1}{l}  c_{l,k-1}\right)\mathbb{E}[\varphi_j]^k,\label{eq:Ephi_ineq}
\end{align}
where $(a)$ follows by interchanging the order of summations, and \eqref{eq:Ephi_ineq} proves the statement of the lemma for $h$. This completes the induction and the proof.
\end{proof}
	
	
\begin{lemma} \label{lemma:appB:004}
	For a homogeneous degree-$2$ polynomial $\varPhi_\poly$ in $\{ \varphi_i : i \le L \}$ with coefficients in $[0,1]$,
	\begin{equation*}
	\mathbb{E}[\varPhi^2_\poly] \le \mathbb{E}[\varPhi_\poly]^2+ 6 k L \mathbb{E} [\varPhi_\poly].
	\end{equation*}
\end{lemma}
\begin{proof}
Let $\varPhi_\poly$ be explicitly given as $\sum_{i^2 \in \bbN_0^2} \alpha_{i^2} \varphi_{i_1} \varphi_{i_2}$. Then,
\begin{equation*}
	\mathbb{E} [\varPhi_\poly^2] = \mathbb{E} \left[ \sum_{i^2 \in \mathbb{N}_0^2} \sum_{j^2 \in \mathbb{N}_0^2} \alpha_{i^2} \alpha_{j^2} \prod_{i \in i^2} \varphi_i \prod_{j \in j^2} \varphi_j \right].
\end{equation*}
Expanding $\varPhi_j$ as $\sum_{x \in \mathcal{X}} \mathbbm{1}_{N_x = j}$ and rearranging the summations,
\begin{align*}
	\mathbb{E} [\varPhi_\poly^2] &= \mathbb{E} \left[ \sum_{i^2 \in \mathbb{N}_0^2} \alpha_{i^2} \sum_{x^2 \in |\mathcal{X}|^2} \left( \prod_{k \in [2]} \mathbbm{1}_{N_{x_k} = i_k} \right) \left( \sum_{j^2 \in \mathbb{N}_0^2} \alpha_{j^2} \prod_{j \in j^2} \left( \sum_y \mathbbm{1}_{N_y = j} \right) \right) \right],\\
	&\le \mathbb{E} \left[ \sum_{i^2 \in \mathbb{N}_0^2} \alpha_{i^2} \sum_{x^2 \in |\mathcal{X}|^2} \underbrace{\left( \prod_{k \in [2]} \mathbbm{1}_{N_{x_k} = i_k} \right)}_{S (x^2,i^2)} \underbrace{\left[\sum_{t = 0,1,2}\sum_{\substack{j^2 \in \mathbb{N}_0^2 \\ |j^2 \cap\ i^2| = t}} \alpha_{j^2}  \prod_{j \in j^2} \left( t + \sum_{y \not\in x^2} \mathbbm{1}_{N_y = j} \right) \right]}_{T (x^2,i^2)} \right].
\end{align*}
Observe that the terms $S (x^2,i^2)$ and $T (x^2,i^2)$ are independent because $S (x^2,i^2)$ depends on $N_{x_1} = i_1$ and $N_{x_2} = i_2$ while $T (x^2,i^2)$ depends on $N_x : x \in \mathcal{X} \setminus x^2$. Therefore, the expectation of their product is equal to the product of their expectations,
\begin{equation} \label{eq:09123}
	\mathbb{E} [\varPhi_\poly^2] = \sum_{i^2 \in \mathbb{N}_0^2} \alpha_{i^2} \sum_{x^2 \in |\mathcal{X}|^2} \mathbb{E} \left[ S (x^2,i^2) \right] \cdot \mathbb{E}\left[ T( x^2,i^2) \right].
\end{equation}
Let us now upper-bound $T (x^2,i^2)$. By its definition,
\begin{align*}
	T (x^2,i^2) &= \sum_{t = 0,1,2}\sum_{\substack{j^2 \in \mathbb{N}_0^2 \\ |j^2 \cap\ i^2| = t}} \alpha_{j^2}  \prod_{j \in j^2} \left( t + \sum_{y \not\in x^2} \mathbbm{1}_{N_y = j} \right) \\
	&\le \sum_{t = 0,1,2}\sum_{\substack{j^2 \in \mathbb{N}_0^2 \\ |j^2 \cap\ i^2| = t}} \alpha_{j^2}  \prod_{j \in j^2} \left( t + \varphi_j \right),
\end{align*}
where the last inequality follows by upper bounding $\sum_{y \in x^2} \mathbbm{1}_{N_y = j}$ by $\varphi_j$. Therefore,
\begin{align}
	T (x^2, i^2) &= \sum_{j^2 \in \mathbb{N}_0^2} \alpha_{j^2} \prod_{j \in j^2} \varphi_j + \sum_{t = 1,2} \sum_{\substack{j^2 \in \mathbb{N}^2_0 \\ |j^2 \cap\ i^2| = t}} \alpha_{j^2} \left( t (\varphi_{j_1} + \varphi_{j_2})+ t^2 \right),\nonumber\\
	&= \varPhi_\poly + \sum_{t = 1,2} \sum_{\substack{j^2 \in \mathbb{N}^2_0 \\ |j^2 \cap\ i^2| = t}} \alpha_{j^2} \left( t (\varphi_{j_1} + \varphi_{j_2})+ t^2 \right),\nonumber\\
	&\le \varPhi_\poly + \sum_{j \in \mathbb{N}_0} \alpha_{i_1 j} (\varphi_{i_1} + \varphi_{j})  + \sum_{j \in \mathbb{N}_0} \alpha_{i_2 j} (\varphi_j + \varphi_{i_2}) + \sum_{j \in \bbN_0} \alpha_{i_1 j} + \sum_{j \in \bbN_0} \alpha_{i_2 j} + 2 (\alpha_{i_1 i_2} + \alpha_{i_2 i_1}),\nonumber\\
	&\overset{(i)}{\le} \varPhi_\poly + 2 |\mathcal{X}| + 2L |\mathcal{X}| + 2L + 2 = \varPhi_\poly + 2 (L+1) (|\mathcal{X}|+1),\label{eq:Tubvar}
\end{align}
where $(i)$ follows from the fact that $\varPhi_\poly$ has coefficients in $[0,1]$, so,
\begin{equation*}
    \alpha_{ij} + \alpha_{ji}
    \begin{cases}
    \le 1, \qquad &\text{if } i,j \le L, \\
    = 0  &\text{otherwise,}
    \end{cases}
\end{equation*}
and from the fact that $\varphi_{i_1}, \varphi_{i_2} \le |\mathcal{X}|$. Plugging \eqref{eq:Tubvar} into \eqref{eq:09123}, we have:
\begin{align*}
    \mathbb{E} [\varPhi_\poly^2] &\le \mathbb{E} [\varPhi_\poly]^2 + 2\mathbb{E} [\varPhi_\poly] (L+1) (|\mathcal{X}|+1),\\
    &\le \mathbb{E} [\varPhi_\poly]^2 + 6 k L \mathbb{E} [\varPhi_\poly],
\end{align*}
where the last inequality follows from the assumptions $ 1 < |\mathcal{X}| \le k$ and $L \ge 1$. 
\end{proof}

Finally, we show Lemma~\ref{lemma:appB:001}, which upper-bounds the conditional moments of $\phi_j$ when $\phi_2 = 0$ and is used in \casetwo.
\begin{lemma} \label{lemma:appB:001}
For all $j \ne 2$ and $h \ge 1$,
\begin{equation*}
\mathbb{E}\left[\varphi_j^h \middle| \varphi_2 = 0 \right] \le \frac{1}{(1-2e^{-2})^{\min (|\mathcal{X}|, h)}} \mathbb{E} [ \varphi_j^h].
\end{equation*}
\end{lemma}
\begin{proof}
By definition of conditional expectation,
\begin{equation}
    \mathbb{E}\left[\varphi_j^h \middle| \varphi_2 = 0\right] = \frac{1}{P (\varphi_2 = 0)}\,\mathbb{E} \left[ \left( \sum_{x \in \mathcal{X}} \mathbbm{1}_{N_x = j} \right)^h \prod_{x \in \mathcal{X}} \mathbbm{1}_{N_x \ne 2} \right]. \label{eq:0912343}
\end{equation}
Without loss of generality, let us denote the domain $\mathcal{X}$ as $\{1,2,\dots |\mathcal{X}| \}$ and set $p_i\triangleq p_x$. Using multinomial expansion for the term $\left( \mathbbm{1}_{N_{1} = j} + \cdots + \mathbbm{1}_{N_{|\mathcal{X}|} = j} \right)^h$ in \eqref{eq:0912343}, and the independence of $N_x$ for different symbols $x \in \mathcal{X}$,
\begin{align*}
    \mathbb{E}\left[\varphi_j^h \middle| \varphi_2 = 0\right] &= \frac{1}{\prod_{x \in \mathcal{X}} P (N_x \ne 2)}\, \mathbb{E}\left[\left(\sum_{ \substack{ \{ h_1,\dots,h_{\mathcal{X}} \} \\ \sum_i h_i = h }} \frac{h!}{h_1 ! h_2 ! \dots h_{|\mathcal{X}|}!} \prod_{i : h_i \ne 0} \mathbbm{1}_{N_{i} = j} \right) \prod_{x \in \mathcal{X}} \mathbbm{1}_{N_x \ne 2} \right],\\
    &= \frac{1}{\prod_{x \in \mathcal{X}} P (N_x \ne 2)}\, \mathbb{E}\left[ \sum_{ \substack{ \{ h_1,\dots,h_{\mathcal{X}} \} \\ \sum_i h_i = h }} \frac{h!}{h_1 ! h_2 ! \dots h_{|\mathcal{X}|}!} \prod_{i : h_i \ne 0} \mathbbm{1}_{N_{i} = j} \prod_{i : h_i = 0} \mathbbm{1}_{N_{i} \ne 2} \right].
\end{align*}
Interchanging the summation and the expectation and again using the independence of $N_x$ between symbols in $\mathcal{X}$,
\begin{align}
    \mathbb{E}\left[\varphi_j^h \middle| \varphi_2 = 0\right] &= \sum_{ \substack{ \{ h_1,\dots,h_{\mathcal{X}} \} \\ \sum_i h_i = h }} \frac{h!}{h_1 ! h_2 ! \dots h_{|\mathcal{X}|}!} \frac{\prod_{i : h_i \ne 0} P (N_{i} = j) }{\prod_{i : h_i \ne 0} P (N_{i} \ne 2)}\nonumber\\
    &\stackrel{(a)}{\le} \sum_{ \substack{ \{ h_1,\dots,h_{\mathcal{X}} \} \\ \sum_i h_i = h }} \frac{h!}{h_1 ! h_2 ! \dots h_{|\mathcal{X}|}!} \frac{\prod_{i : h_i \ne 0} P (N_{i} = j) }{\prod_{i : h_i \ne 0} (1-2e^{-2})}\nonumber\\
    &\stackrel{(b)}{\le}\frac{1}{(1-2e^{-2})^{\min (|\mathcal{X}|, h)}}\sum_{ \substack{ \{ h_1,\dots,h_{\mathcal{X}} \} \\ \sum_i h_i = h }} \frac{h!}{h_1 ! h_2 ! \dots h_{|\mathcal{X}|}!} \prod_{i : h_i \ne 0} P (N_{i} = j)\nonumber\\
    &=\frac{1}{(1-2e^{-2})^{\min (|\mathcal{X}|, h)}} \mathbb{E}\left[ \sum_{ \substack{ \{ h_1,\dots,h_{\mathcal{X}} \} \\ \sum_i h_i = h }} \frac{h!}{h_1 ! h_2 ! \dots h_{|\mathcal{X}|}!} \prod_{i : h_i \ne 0}  \mathbbm{1}_{N_{i} = j} \right] \nonumber\\
    &= \frac{1}{(1-2e^{-2})^{\min (|\mathcal{X}|, h)}} \mathbb{E} [ \varphi_j^h],
\end{align}
where $(a)$ follows because $P(N_i \ne 2) = 1 - \frac{1}{2} (np_i)^2 e^{-np_i} \ge 1-2e^{-2}$, and $(b)$ follows because $|{i : h_i \ne 0}|\le h$ since $\sum_{i=1}^h h_i = h$ with $h_i \ge 1$ if non-zero. 
\end{proof}	

\end{document}